\documentclass[12pt]{amsart}
\usepackage[T1]{fontenc}
\usepackage{lmodern,microtype,amsmath,amssymb,amsthm,mathtools}
\usepackage[margin=1in]{geometry}
\usepackage{needspace}
\usepackage{quiver}
\usepackage{mathrsfs}
\usepackage[colorlinks=true,linkcolor=blue,citecolor=blue,urlcolor=blue]{hyperref}
\newtheorem{theorem}{Theorem}[section]
\newtheorem{lemma}[theorem]{Lemma}
\newtheorem{proposition}[theorem]{Proposition}

\theoremstyle{definition}
\newtheorem{definition}[theorem]{Definition}
\newtheorem{remark}[theorem]{Remark}
\newtheorem{quest}[theorem]{Question}
\newtheorem{Conj}{Conjecture}
\newcommand{\R}{\mathbb R}
\newcommand{\Z}{\mathbb Z}\newcommand{\C}{\mathbb C}
\newcommand{\T}{\mathbb T}\newcommand{\N}{\mathbb N}
\newcommand{\XX}{\mathrm{X}}

\newcommand{\GenNz}{\mathsf{Nil}^{\mathrm{gen},0}}
\newcommand{\I}{\mathcal{I}}
\newcommand{\NN}{\mathsf{Nil}}
\newcommand{\MCS}{\mathrm{MCS}}
\newcommand{\GenN}{\mathsf{Nil}^{\mathrm{gen}}}
\newcommand{\GenI}{\mathcal{I}^{\mathrm{gen}}}
\newcommand{\Null}{\mathsf{Null}}

\newcommand{\lalg}{\mathfrak l}
\newcommand{\YY}{\mathrm{Y}}
\newcommand{\ZZ}{\mathrm{Z}}
\title{A structure theorem for multiple correlation sequences on nilsystems}
\author[O. Shalom]{Or Shalom}
\address{Department of Mathematics\\ Bar Ilan University \\ 
Ramat Gan \\
5290002, Israel}
\email{Or.Shalom@math.biu.ac.il}
\begin{document}
\begin{abstract}
In a recent paper \cite{FKshalom} the author refuted several conjectures about the vanishing of multiple correlation sequences. The goal of this paper is to provide a counterpart for the conjectures of Frantzikinakis--Kuca and of Leibman. The significance of this manuscript is twofold: firstly, we provide a structure theorem for multiple correlation sequences on nilsystems along arithmetic progressions (in the sense of Frantzikinakis--Kuca), and a corresponding structure theorem for polynomial orbits of subnilmanifolds in connected nilmanifolds (in the sense of Leibman). Secondly, we identify \emph{null generalized nilsequences}, as a notion of \emph{smallness}, which provides a structured counterpart for null sequences. In particular, in the special case of ergodic nilsystems, we answer positively a question of Frantzikinakis on the structure of multiple correlation sequences.
\end{abstract}
\maketitle
\section{Introduction}
In \cite{furstenberg1977ergodic}, Furstenberg gave an ergodic-theoretic proof for Szemer\'edi's theorem about the existence of arbitrarily long arithmetic progressions in sets of positive upper Banach density. The ergodic-theoretic version of Szemer\'edi's theorem is of the following form:
\begin{theorem}[Furstenberg-Szemer\'edi's theorem] Let $\XX=(X,\mathcal{B},\mu,T)$ be a standard Borel measure-preserving system and let $A\in \mathcal{B}$ a set of positive measure. For every $\ell\in \mathbb{N}$ let
\begin{equation}\label{multiplerecurrence}
a_\ell(n) \coloneqq \mu (A\cap T^{-n}A\cap\dots \cap T^{-n\ell}A) = \int_X 1_A(x)\cdot 1_A(T^nx)\cdots 1_A(T^{n\ell}x)\,d\mu(x).
\end{equation}
Then the Ces\'aro averages of $a_\ell$ have a positive lower limit
$$\liminf_{N\rightarrow\infty}\frac{1}{N}\sum_{n=1}^N a_\ell(n)>0.$$
\end{theorem}
It is now classical that this lower limit is in fact a limit, see \cite{host2005nonconventional,ziegler2007universal} and moreover formulae for the said limit also exist, see \cite{BHK,zieglerformula}. The work of Furstenberg led to numerous generalizations of Szemer\'edi's theorem. For instance, several cases for polynomial progressions were established \cite{bergelson1996polynomial,Leibman3,Frk052,FrK05}; for Hardy sequences e.g., \cite{Fran-Hardy}; for IPs \cite{furstenberg1985ergodic,BergelsonIPSzemeredi,nilpotentIPszemeredi}; for more general group actions \cite{furstenberg1978ergodic,tao2008norm,austin2016nonconventional,robertson2016characteristic,ShalomGrothendieck}; for improvements on the lower bounds see \cite{Franpoly,abb,shalom2,abs,krashalom2025}; and for combinatorial lines see \cite{DHJ3,DHJ}.

The sequence $a_\ell$ in \eqref{multiplerecurrence} is an example of a larger class of sequences called $\ell$-fold \emph{multiple correlation sequences} (here along linear iterates). 
\begin{definition}[Multiple correlation sequences]
    Let $\ell\geq 1$ and let $\XX=(X,\mathcal{B},\mu,T)$ be an ergodic measure-preserving system. An \emph{$\ell$-fold multiple correlation sequence on $\XX$} is any sequence of the form 
    \begin{equation}\label{MCScommutative}
    a(n) = \int_X f_0(x)\cdot \prod_{i=1}^{\ell} f_i(T^{in}x)\,d\mu(x)
\end{equation}
where $f_0,\dots,f_{\ell}\in L^\infty(\mu)$ are bounded functions. We denote by $\MCS_{\ell}(\XX)$ the class of all multiple correlation sequences on $\XX$.
\end{definition}

Bochner--Herglotz theorem, together with a simple polarization trick, fully classifies single (bilinear) correlation sequences. Indeed, it implies that any $a\in \MCS_1(\XX)$ admits a complex-valued measure $\mu_a$ on the circle, with bounded variation, so that 
\begin{equation}\label{BochnerHerglotz}a(n) = \int_{S^1} t^n~\,d\mu_a(t).
\end{equation}
Moreover, the several properties of the measure $\mu_a$ directly correspond to the system. For instance, $\mu_a$ is atomic at $\lambda$ only if $\lambda$ is an eigenvalue.

The goal of this paper is to extend this characterization for multiple correlation sequences, which we establish in the special case where $\XX$ is a \emph{nilsystem}. In order to formulate our main results formally, we first define some notations and survey related results.

\subsection{Families of sequences}
Let $k\geq 1$. A \emph{$k$-step nilsystem} is a measure-preserving system of the form $\XX=(G/\Gamma,\mathcal{B},\mu_{G/\Gamma},R_\alpha)$, where $G$ is a $k$-step nilpotent Lie group, $\Gamma$ is a discrete co-compact subgroup, $\mathcal{B}$ is the Borel $\sigma$-algebra, $\mu_{G/\Gamma}$ is a normalized $G$-invariant measure and finally $R_\alpha g\Gamma = \alpha g\Gamma$ is a translation by some $\alpha \in G$. It is well known that a nilsystem is ergodic if and only if all orbits are dense and in that case it is uniquely ergodic. Furthermore, ergodic nilsystems admit a presentation for $G$ and $\Gamma$ such that $G=\langle G^0,\alpha\rangle$ where $G^0$ is the connected component of the identity in $G$. Unless mentioned otherwise, we shall assume throughout that all ergodic nilsystems are of this form.

Most of the following definitions are from \cite{Fran}. 
\begin{definition}[Generalized and integral combinations of nilsequences, and null sequences] \label{families:def} Let $k\geq 1$. We have the following notions of \emph{structured sequences}:
   \begin{itemize}
       \item[(1)] A \emph{basic $k$-step nilsequence} is a sequence of the form $n\mapsto F(\alpha^n x_0)$ where $F:G/\Gamma\rightarrow \C$ is a continuous function on a nilsystem $(G/\Gamma,R_\alpha)$ and $x_0\in G/\Gamma$ is a point. We write $\NN_k\subseteq \ell^\infty(\mathbb{Z})$ for the closed linear subspace of $\ell^\infty(\mathbb{Z})$ generated by the basic $k$-step nilsequences. Given a basic nilsequence $\psi(n)=F(\alpha^n x_0)$ as before, we refer to $F$ as a \emph{symbol}\footnote{The symbol is not unique.} of $\psi$.
       \item[(2)] A basic $k$-step \emph{generalized} nilsequence is a sequence of the form $n\mapsto F(\alpha^n x_0),$ where $F:G/\Gamma\rightarrow\C$ is a Riemann-integrable\footnote{For the purpose of this paper, we say that $F$ is Riemann integrable if it is bounded and the set of discontinuities of $F$ is Haar-null.} function on an ergodic\footnote{ergodicity was not originally assumed in the literature \cite{Fran}, but it is necessary as otherwise any bounded sequence is a $1$-step generalized nilsequence.} nilsystem $(G/\Gamma,R_\alpha)$. We write $\GenN_k\subseteq \ell^\infty(\Z)$ for the closed linear subspace of $\ell^\infty(\Z)$ generated by the basic $k$-step generalized nilsequences. 
       \item[(3)] An \emph{integral combination} of $k$-step nilsequences is a sequence of the form $n\mapsto \int_S \psi_s(n)\,d\nu(s) $ where $S$ is a compact metric space, $\nu$ is a complex-valued Borel complex measure with bounded variation, $\psi_s\in \NN_k$ for all $s\in S$ and $s\mapsto \psi_s(n)$ is uniformly bounded in $n$ and $s$ and measurable for all $n\in \mathbb{Z}$. We write $\I_k\subseteq \ell^\infty(\Z)$ for the closed linear subspace of $\ell^\infty(\Z)$ generated by all the integral combinations of $k$-step nilsequences. Similarly, we define $\GenI_k$ with respect to generalized nilsequences.
    \end{itemize}
    We also have the following notions of \emph{small} sequences.
    \begin{itemize}
       \item[(4)] We say that a bounded sequence $a\in \ell^\infty(\Z)$ is $\ell^p$-null if $$\lim_{N\rightarrow\infty}\frac{1}{N}\sup_{M\in \Z} \sum_{n=M}^{N+M}|a(n)|^p=0$$ with the convention that $a$ is $\ell^\infty$-null if $\lim_{|n|\rightarrow\infty}a(n)=0.$ We denote by $\Null_p$ the set of all $\ell^p$-null sequences. Thus $\Null_\infty=c_0$ is the space of sequences tending to zero, and we use the two notations interchangeably.
       \item[(5)] We denote by $\GenNz_k$ the set of all basic generalized $k$-step nilsequences which admit a symbol $F$ that vanishes outside of a null set.
   \end{itemize} 
\end{definition}
It is not hard to prove that every sequence in $\GenNz_k$ belongs to $\Null_p$ for every $1\leq p<\infty$. Indeed, let $\psi(n)=F(\alpha^nx_0)$ with $F$ as in (5) and let $\varepsilon>0$. Since $|F|^p$ is Riemann-integrable and $\int_{G/\Gamma}|F|^p\,d\mu_{G/\Gamma}=0$, there is a continuous $g\geq|F|^p$ with $\int_{G/\Gamma} g\,d\mu_{G/\Gamma}<\varepsilon$. Since ergodic nilsystems are uniquely ergodic, $\frac{1}{N}\sum_{j=0}^{N}g(\alpha^jx)\rightarrow\int g\,d\mu_{G/\Gamma}$ uniformly in $x\in G/\Gamma$. Applying this with $x=\alpha^Mx_0$ gives $\limsup_{N\rightarrow\infty}\frac{1}{N}\sup_{M\in\Z}\sum_{n=M}^{M+N}|\psi(n)|^p\leq\varepsilon$.\\

The following lemma summarizes some simple inclusions that will be used throughout.
\begin{lemma}\label{inclusions}
    Let $k\geq 1$. The following inclusions hold:
    \begin{itemize}
        \item[(i)] $\NN_k\subseteq \NN_{k+1}$, $\GenN_k\subseteq \GenN_{k+1}$, $\I_k\subseteq \I_{k+1}$ and $\GenI_k\subseteq \GenI_{k+1}.$
        \item[(ii)] $\NN_k\subseteq \GenN_k\cap \I_k$ and both $\GenN_k$ and $\I_k$ are subsets of $\GenI_k.$
        \item[(iii)] $\Null_\infty \subseteq \Null_p\cap \I_1$ for all $p\geq 1$.
        \item[(iv)] If $\XX$ is a $k$-step nilsystem, then $\MCS_{\ell}(\XX)\subseteq \I_k$ for all  $\ell\geq 1$.
    \end{itemize}
\end{lemma}
\begin{proof}
    Property $(i)$ follows immediately from the convention that a $k$-step nilsystem is also a $(k+1)$-nilsystem. The inclusions $\NN_k\subseteq\GenN_k$ and $\mathcal{I}_k\subseteq\GenI_k$ hold since continuous functions are Riemann integrable and Leibman's orbit-closure Lemma (see Lemma~\ref{lem:leibman-orbit}) guarantees that all orbit closures are ergodic nilsystems. The other inclusions in $(ii)$ follow from the fact that an integral combination of nilsequences may be taken with respect to a Dirac measure $\nu$. The inclusion $\Null_\infty\subseteq\Null_p$ is immediate. For the inclusion $\Null_\infty\subseteq \mathcal{I}_1$, take some $a\in \Null_\infty$. Then for every $\varepsilon>0$, the set $E_\varepsilon(a)\coloneqq \{n\in \mathbb{Z}: |a(n)|>\varepsilon\}$ is finite. Put $\psi_t(n)=e^{2\pi i nt}$ for all $t\in \R/\Z$ and note that $\psi_t$ is a basic $1$-step nilsequence (defined on the $1$-step nilsystem $(\T,R_t)$). Now, set
    $$\omega_\varepsilon(t) \coloneqq \sum_{m\in E_\varepsilon} a(m)\cdot \overline{\psi_t(m)}$$ and observe that for every $n\in \Z$ we have,
    $$\int_\T \omega_\varepsilon(t)\psi_t(n)\,dt=\sum_{m\in E_\varepsilon}a(m)\cdot\int_{\T}\psi_t(n)\cdot\overline{\psi_t(m)}\,dt = a(n)\cdot \mathbf{1}_{E_\varepsilon}(n).$$ Taking $\varepsilon\rightarrow 0$, we see that $a\cdot \mathbf{1}_{E_\varepsilon}\rightarrow a$ uniformly. This shows that $a\in \I_1$, giving $(iii)$. Finally, to prove $(iv)$, let $\XX$ be a $k$-step nilsystem and $a\in \MCS_{\ell}(\XX)$. Write $f_0,f_1,\dots,f_{\ell}$ for the bounded functions associated with $a$. Since $C(\XX)$ is dense in $L^2(\mu)$, for every $\varepsilon>0$ we can find continuous $\tilde{f}_i\in C(\XX)$ with $\|\tilde f_i\|_{L^\infty}\leq \|f_i\|_{L^\infty}$ and $\|f_i-\tilde{f}_i\|_{L^2}<\varepsilon$. Let $\tilde{a}$ denote the multiple correlation sequence obtained by replacing $f_i$ with $\tilde{f}_i$ for all $i=0,\dots,\ell$. Applying the Cauchy--Schwarz inequality iteratively, we see that $\|a-\tilde{a}\|_{\ell^\infty} =O(\varepsilon)$. Now $\tilde a\in \I_k$ by definition: take $S=G/\Gamma$, take $\nu=\tilde f_0\,d\mu_{G/\Gamma}$, a complex Borel measure of bounded variation, and for $x\in S$ take $$\psi_x(n)=\prod_{i=1}^{\ell}\tilde f_i(\alpha^{in}x) = F(\boldsymbol\alpha^{n}\mathbf x),\qquad \boldsymbol\alpha=(\alpha,\alpha^2,\dots,\alpha^{\ell})\in G^{\ell},\quad \mathbf x=(x,\dots,x),$$ where $F=\tilde f_1\otimes\dots\otimes \tilde f_{\ell}$ is continuous on the $k$-step nilmanifold $(G/\Gamma)^{\ell}$. Thus $\psi_x\in \NN_k$ for every $x$, the family is uniformly bounded, and $x\mapsto \psi_x(n)$ is continuous. Taking $\varepsilon\rightarrow 0$ we conclude that $a\in \I_k$, as required.
\end{proof}
\subsection{The nil plus null decomposition}
Current literature on multiple correlation sequences predominantly focuses on establishing results of the \emph{nil plus null} form. The most classical form of this type is due to Bergelson, Host and Kra \cite{BHK}.
\begin{theorem}[Bergelson--Host--Kra]\label{thm:BHK}
    For every $\ell\geq 1$ and every ergodic measure-preserving system $\XX=(X,\mathcal{B},\mu,T)$, we have
    $$\MCS_{\ell}(\XX)\subseteq \NN_\ell +\Null_1.$$
\end{theorem}
Many generalizations of this result exist in the literature (see \cite{Leibman1,Leibman2,integerpart,MC4,decomposition,Frankucajoint,MCS,subsequences,MC3,morejoint,ShalomGrothendieck}, among many others). We emphasize one generalization that is relevant to this paper. In \cite{Leibman1} Leibman extended the structural result of Bergelson, Host and Kra to polynomial iterates. He proved
\begin{theorem}[Polynomial nil plus null decomposition]
    Let $\XX=(X,\mathcal{B},\mu,T)$ be an ergodic invertible measure-preserving system and $f_0,\dots,f_\ell\in L^\infty(\mu)$. Let $p_1,\dots,p_{\ell}\in \Z[n]$ be polynomials with integer coefficients. Then,
    $$n\mapsto \int_X f_0(x)\cdot\prod_{i=1}^{\ell} f_i(T^{p_i(n)}x)\, d\mu$$ is a sum of a nilsequence and a null sequence.
\end{theorem}
The distinction between this and the theorem of Bergelson, Host and Kra is that here the step of the nilsequence is not specified, but it is shown to be bounded by a constant depending only on $\ell$ and the polynomials $p_1,\dots,p_{\ell}$. All of these results may be interpreted as follows: every multiple correlation sequence is equal to a \emph{structured sequence} (here a nilsequence) up to a \emph{small error} (here a null sequence). However, unfortunately, null sequences are far less understood than nilsequences. Since there exist null sequences that are not multiple correlation sequences, we can ask what else can be said about this error term. Finding the right notions of structure and smallness in this setting is related to a fundamental question of Frantzikinakis \cite[Problem 1]{Fran}. The question is formalized in great generality, but we focus on the case of a single transformation for the sake of simplicity.
\begin{quest}[Frantzikinakis' question]\label{Fran:quest}
    Let $\ell\geq 1$ and let $\XX=(X,\mathcal{B},\mu,T)$ be an ergodic measure-preserving system. Is it true that
    $\MCS_{\ell}(\XX) \subseteq \GenI_{\ell}\, ?$
\end{quest}
Bri\"et and Green \cite{BrietGreen} showed that Question~\ref{Fran:quest} has a negative answer if $\GenI_\ell$ is replaced by $\I_\ell$: already for $\ell=2$ there is an ergodic system $\XX$ with $\MCS_2(\XX)\not\subseteq\I_2$. On the other hand, the question has a positive answer in the following special cases: $\ell=1$ (by the Herglotz--Bochner theorem), nilsystems of step at most $\ell$ (Lemma~\ref{inclusions}(iv)), and systems that are mixing of all orders, such as Bernoulli systems. The next natural case is that of general nilsystems, where for $\ell=1$ the answer is known. The following result was established, first by Stepin \cite{Stepin} for $G$ connected and simply connected, then by Host, Kra and Maass \cite{HKM} for $2$-step general (not necessarily connected) nilsystems, and for all ergodic nilsystems by Ackelsberg, Richter and the author \cite{ARS} and independently by Frantzikinakis and Kuca \cite[Theorem 3.7]{FranKuca}.
\begin{theorem}[Riemann--Lebesgue property for correlation sequences]\label{thm:ell1}
    Let $G/\Gamma$ be an ergodic nilsystem and $f,g\in L^\infty(\mu_{G/\Gamma})$, one of which is orthogonal to the abelian factor\footnote{Here we assume the standard presentation, i.e., $G=\langle G^0,\alpha\rangle$} $G/G_2\Gamma$. Then $$\lim_{n\rightarrow\infty}\int_{G/\Gamma} f(\alpha^nx)\cdot g(x)\, d\mu(x)= 0.$$
\end{theorem}
Many of the proofs mentioned above establish this fact by proving that the orthogonal complement in $L^2(\mu_{G/\Gamma})$ of the abelian factor (which corresponds to $L^2(\mu_{G/G_2\Gamma})$ under natural assumptions) is of Lebesgue maximal spectral type, so that the Riemann--Lebesgue lemma applies. The following conjecture of Frantzikinakis and Kuca can be interpreted as a higher order form of a Riemann--Lebesgue lemma for nilsystems.
\begin{Conj}[Frantzikinakis--Kuca {\cite[Conjecture 4]{FranKuca}}]\label{FK:conj}
    Let $(G/\Gamma,\mu_{G/\Gamma},R_\alpha)$ be an ergodic nilsystem and $k_1,\dots,k_\ell\in \Z$ non-zero and distinct. Then
    $$\lim_{n\rightarrow\infty}\int_{G/\Gamma}f_0(x)\cdot f_1(\alpha^{k_1n}x)\cdots f_\ell(\alpha^{k_\ell n}x)\,d\mu_{G/\Gamma}(x)=0$$
    whenever $f_j$ is orthogonal to the $\ell$-step factor (i.e., $G/G_{\ell+1}\Gamma$ under the presentation $G=\langle G^0,\alpha\rangle$) for some $j$.
\end{Conj}
Leibman also proposed a related conjecture \cite{Leibman1} for polynomial iterates. Let $G$ be a nilpotent group. We say that a sequence $g:\Z\rightarrow G$ is a \emph{polynomial sequence} if
\begin{equation}\label{eq:polynomial}
    g(n) = a_1^{p_1(n)}\cdots a_\ell^{p_\ell(n)}
\end{equation}
where $a_1,\dots, a_{\ell}\in G$ and $p_1,\dots,p_{\ell}$ are integer-valued polynomials.
\begin{Conj}[Leibman's conjecture]\label{Leibman:conj}
Let $X=G/\Gamma$ be a connected nilmanifold, let $\pi:G\rightarrow X$ denote the quotient map and let $Y=\pi(H)$, where $H$ is a connected rational closed subgroup of $G$. Let $g:\Z\rightarrow G$ be a polynomial sequence with $g(0)=1_G$ such that $g(\Z)Y = \bigcup_{n\in \Z} g(n)Y$ is dense in $X$, and assume that $G=\langle G^0,(g(n))_{n\in \Z}\rangle$. Let $Z$ be the \emph{normal closure} of $Y$ in $X$. Then,
$$\lim_{n\rightarrow\infty}\left( \int_{g(n)Y} F\, d\mu_{g(n)Y} - \int_{g(n)Z} F\,d\mu_{g(n)Z}\right) =0,$$ for every continuous $F\in C(X)$, where $\mu_{g(n)Y}$ and $\mu_{g(n)Z}$ are the Haar measures on $g(n)Y$ and $g(n)Z$, respectively.
\end{Conj}
In this conjecture, the normal closure of $Y$ is defined to be the nilmanifold $Z=\pi(\langle\langle H \rangle\rangle)$ where $\langle\langle H \rangle\rangle$ is the normal closure of $H$ in $G$. 

These conjectures were recently refuted by the author in \cite{FKshalom}. However, the counterexample, which was derived via the theory of Fourier analysis on nilmanifolds and symbolic computations, does not provide a negative answer to Question~\ref{Fran:quest}. In fact, it is a $1$-step generalized nilsequence.

The purpose of this paper is to provide an alternative to both of these conjectures, and provide a positive answer to Question~\ref{Fran:quest} in a special case. Our main finding is that null generalized nilsequences (i.e., elements in $\GenNz_{\ell}$) are a candidate for a structured alternative to the general null sequences. 
\begin{theorem}[Structure theorem for multiple correlation sequences on nilsystems]\label{FK:thm}
    Let $\ell\geq 1$ and let $\XX=(G/\Gamma,\mathcal{B},\mu_{G/\Gamma},R_\alpha)$ be an ergodic nilsystem of any step. Then,
    $$\MCS_{\ell}(\XX)\subseteq \NN_{\ell} + \GenNz_{\ell} \subseteq \GenI_{\ell}.$$ 
\end{theorem}
\begin{remark}
The same argument applies verbatim to the more general convention of Frantzikinakis and Kuca where the standard times for the multiple correlation sequence $n,2n,\dots,\ell n$ are replaced with arbitrary distinct times $k_1n,k_2n,\dots,k_\ell n$.
\end{remark}
Moreover, we provide the following alternative to Leibman's conjecture.
\begin{theorem}[Structured equidistribution in connected nilsystems]\label{leibman:thm}
    In the setting of Conjecture~\ref{Leibman:conj}, the difference $$b_f(n)\coloneqq \left(\int_{g(n)Y} f\, d\mu_{g(n)Y} - \int_{g(n)Z} f\,d\mu_{g(n)Z}\right)$$ is of the form $b_f(n)=B_f(S^n x)$, where $B_f$ is a Riemann-integrable map on an auxiliary ergodic nilsystem $(\Omega,S,\mu_{\Omega})$ and $x\in \Omega$, such that $B_f\equiv 0$, $\mu_{\Omega}$-almost everywhere. In particular, $b_f\in \GenNz_{s}$ for some $s\in \N$.
\end{theorem}
\subsection{Outline of the proof}
We outline the key aspects of the proofs, starting with Theorem~\ref{leibman:thm}. In Lemma~\ref{auxiliary:lemma} we identify a key auxiliary ergodic nilsystem $(\Omega,S)$, on which the generalized nilsequence will be defined. Lemma~\ref{control:lem} classifies exactly when a sequence lies in $\GenNz_s$ with respect to $\Omega$. It gives a dynamical criterion for belonging to $\GenNz_s$ that can be verified \emph{locally}. This paves the way to the main analytic argument, Theorem~\ref{control:thm}, which establishes this criterion. It relies on a simple fact about zeros of analytic functions (Lemma~\ref{zeros}) and on the quantitative equidistribution theorem of Green and Tao (Proposition~\ref{approximation:prop}).

Theorem~\ref{FK:thm} requires an additional argument in order to control the step $s$ of the error term. To do so, we construct $\Omega$ from the $\ell$-step factor $Z_{\ell}(X)$. Nilsystems of step at most $\ell$ are handled directly by Theorem~\ref{leibman:thm}. We then establish the result under two assumptions: (1) $X$ is connected, and (2) $X$ has step $\ell+1$. We approximate the functions $f_0,\dots,f_\ell$ by nilcharacters and use a product trick: $f_i(x)\overline{f_i(y)}$ is invariant under the diagonal action of $G_{\ell+1}$, so it descends to the quotient of $X\times X$ by this action, in which the image of the diagonal is a copy of $Z_\ell(X)$. We then obtain the claim by the same argument as in Theorem~\ref{leibman:thm}, this time with $s=\ell$.

The reduction to the connected case is a standard argument, relying on the fact that a nilsystem has finitely many connected components. The reduction from higher step to step $\ell+1$ follows from the degree lowering method of Frantzikinakis and Kuca \cite{FranKuca}, and requires a technical argument of passing from weak to norm convergence (Lemma~\ref{lem:weak-to-norm}).
\subsection*{Acknowledgments}
The author is supported by the Alon Fellowship and the Polak foundation.
\section{Preliminaries}
\subsection{The standard presentation and rationality}
As mentioned before, an ergodic nilsystem $G/\Gamma$ admits a presentation where $G=\langle G^0,\alpha\rangle$. We may further assume, by passing to the universal cover, that $G^0$ is simply connected and that the groups of the lower central series $G_j$ are all closed, connected and rational for $j\geq 2$. Under these assumptions, the $j$-th Host--Kra factor satisfy $Z_j(G/\Gamma) = G/G_{j+1}\Gamma$ (see \cite{hk-book}). 
\subsection{Analytic nilmanifolds}
Here we survey some definitions and basic properties of analytic nilmanifolds.
\begin{definition}[Real analytic function]\label{def:real-analytic}
    Let $m\geq 0$ and let $V\subseteq \mathbb{R}^m$ be an open domain. A function $f:V\rightarrow\C$ is called \emph{real analytic} if for every $v_0\in V$, there exist $r>0$ and coefficients $(c_b)_{b\in \N_0^m}\subseteq \C$ such that
    $$f(v) = \sum_{b\in \N_0^m} c_b(v-v_0)^b$$ with the convention that $(v-v_0)^b = \prod_{i=1}^m (v_i-v_{0,i})^{b_i}$ and the series converges absolutely for every $v\in V$ with $|v-v_0|<r$. More generally, we say that a map taking values in $\C^k$ is real analytic if each of its coordinates is real analytic.
\end{definition}
The following lemma is basic (see e.g., \cite{KrantzParks}).
\begin{lemma}[Basic properties of real analytic functions]
    Let $V\subseteq \R^m$ be an open domain.
    \begin{itemize}
        \item[(i)] Every real analytic function $f:V\rightarrow \C$ is smooth and the coefficients in Definition~\ref{def:real-analytic} are given by $c_b = \frac{\partial^b f}{b!}(v_0).$ Here, $\partial^b\coloneqq\frac{\partial^{b_1}}{\partial v_1^{b_1}}\cdots\frac{\partial^{b_m}}{\partial v_m^{b_m}}\quad\text{and}\quad b!\coloneqq b_1!\cdots b_m!$. 
        \item[(ii)] Sums, products and compositions (whenever they make sense) of real analytic functions are real analytic.
    \end{itemize}
\end{lemma}
A manifold is locally homeomorphic to $\R^m$. A real analytic structure is a choice of such local homeomorphisms whose changes of coordinates are real analytic. Coordinate charts make this precise.
\begin{definition}[Coordinate charts and real analytic manifolds]
Let $M$ be a Hausdorff second countable topological space and let $m\geq 0$. An ($m$-dimensional) \emph{coordinate chart} on $M$ is a pair $(U,\phi)$ where $U\subseteq M$ is open and $\phi:U\rightarrow \phi(U)$ is a homeomorphism onto an open subset $\phi(U)\subseteq \R^m$. We refer to $U$ as the \emph{coordinate neighborhood}. We say that two charts $(U,\phi)$ and $(U',\phi')$ are analytically compatible if the transition map $$\phi'\circ \phi^{-1}:\phi(U\cap U')\rightarrow \phi'(U\cap U')$$ and its inverse are real analytic. A \emph{real analytic atlas} on $M$ is a family of pairwise analytically compatible charts whose domains cover $M$. Suppose that $W\subseteq M$ is open, we say that a function $f:W\rightarrow \C$ is \emph{real analytic} if $f\circ\phi^{-1}$ satisfies Definition~\ref{def:real-analytic} on $\phi(U\cap W)$ for every chart $(U,\phi)$ of $M$. It suffices to check the latter for a single atlas, see e.g., \cite[Chapter~I]{Hel78}. Every real analytic atlas is contained in a unique maximal one, called a \emph{real analytic structure}. $M$ together with a real analytic structure is called an ($m$-dimensional) \emph{real analytic manifold}, and by a chart of $M$ we mean a chart of its maximal atlas.
\end{definition}
Next we define relationships between analytic manifolds.
\begin{definition}
    Let $M$ and $M'$ be real analytic manifolds of dimensions $m$ and $m'$, respectively. 
    \begin{itemize}
        \item[(1)] A continuous map $F:M\rightarrow M'$ is \emph{analytic} if $\phi'\circ F\circ \phi^{-1}$ is real analytic on $\phi(U\cap F^{-1}(U'))$ for all charts $(U,\phi)$ of $M$ and $(U',\phi')$ of $M'$. An \emph{analytic diffeomorphism} is an analytic bijection whose inverse is analytic. A \emph{local analytic diffeomorphism} is an analytic map $F$ such that every $x\in M$ admits an open neighborhood $U$ for which $F(U)$ is open and $F|_U:U\rightarrow F(U)$ is an analytic diffeomorphism.
        \item[(2)] A subset $Y\subseteq M$ is called an \emph{embedded analytic submanifold of dimension $k$} if for every $y\in Y$ there exists a chart $(U,\phi)$ of $M$ with $y\in U$ and $$\phi(U\cap Y) = \phi(U)\cap \bigl(\R^k\times\{0^{m-k}\}\bigr).$$ In this case, the maps $\mathrm{p}_k\circ\phi|_{U\cap Y}$, where $\mathrm{p}_k:\R^m\rightarrow\R^k$ is the projection onto the first $k$ coordinates, form a real analytic atlas on $Y$ (with the subspace topology), and we always equip $Y$ with this structure.
    \end{itemize}
\end{definition}
We are specifically interested in manifolds arising as a quotient of a Lie group by a lattice. These are always analytic.
\begin{proposition}[Analytic structure of Lie groups and their quotients]\label{lieisanalytic:prop}
    Let $G$ be a Lie group with countably many connected components (e.g., if $G/G^0$ is finitely generated). Then,
    \begin{itemize}
        \item[(i)] $G$ admits a unique real analytic structure, compatible with its topology, for which multiplication and inversion are analytic. Moreover, with respect to these structures every continuous homomorphism between Lie groups is analytic.
        \item[(ii)] (Closed subgroup theorem) Every closed subgroup $H\leq G$ is an embedded analytic submanifold of $G$. In particular, it is a Lie group.
        \item[(iii)] If $\Gamma\leq G$ is a discrete subgroup, then $G/\Gamma$ also admit a unique real analytic structure for which the quotient map $\pi:G\rightarrow G/\Gamma$ is a local analytic diffeomorphism. 
    \end{itemize}
\end{proposition}
\begin{proof}
    See \cite[Chapter~II, \S\S2 and~4]{Hel78}. For (iii), note that $\Gamma$ acts freely and properly discontinuously on $G$ by right translations, so $\pi$ is a covering map, and the local inverses of $\pi$, composed with charts of $G$, form an analytic atlas on $G/\Gamma$ (the transition maps are locally right translations by elements of $\Gamma$).
\end{proof}
\subsubsection{Lifts}
\begin{definition}
    Let $p:\tilde{Y}\rightarrow Y$ be a covering map between topological spaces. Let $f:Z\rightarrow Y$ a continuous map. A \emph{lift} of $f$ (with respect to $p$) is a continuous map $\tilde{f}:Z\rightarrow \tilde{Y}$ such that $p\circ \tilde{f}=f$.
\end{definition}
The following lemma will be useful.
\begin{lemma}[Analytic lifts]\label{analyticlift:lem}
Let $p:\tilde{Y}\rightarrow Y$ be a covering map between real analytic manifolds which is also a local analytic diffeomorphism. Let $Z$ be a connected and simply connected real analytic manifold and let $f:Z\rightarrow Y$ analytic. Then for every $z_0\in Z$ and every $\tilde{y}_0\in p^{-1}(f(z_0))$, there exists a unique lift $\tilde{f}$ of $f$ with $\tilde{f}$ analytic and $\tilde{f}(z_0)=\tilde{y_0}$. 
\end{lemma}
\begin{proof}
    The existence and uniqueness is the unique lifting property (See e.g., \cite[Propositions~1.33 and~1.34]{Hatcher}). To prove that $\tilde{f}$ is analytic, let $z\in Z$ and choose some open neighborhood $U$ of $\tilde{f}(z)$ such that $p|_U:U\rightarrow p(U)$ is an analytic diffeomorphism. Then $\tilde{f} = (p|_U)^{-1}\circ f$ on the open set $z\in \tilde{f}^{-1}(U)$ is analytic.
\end{proof}

\subsection{Corollaries of Leibman's work}
We use intensively the following results from Leibman's work \cite{Leibmanpointwise}. The first asserts that polynomial sequences can be realized as orbits of a unipotent automorphism on a larger nilmanifold.
\begin{lemma}[Realizing a polynomial sequence as an automorphism on a larger nilmanifold]\label{Leibman:realization}
    Let $N$ be a nilpotent Lie group, let $N^0$ denote the connected component of the identity and suppose that $N/N^0$ is finitely generated. Let $\Lambda$ be a lattice for $N$ and $g:\Z\rightarrow N$ be a polynomial sequence. There exist a nilpotent Lie group $\tilde{N}$ of (possibly higher) finite step, such that $\tilde{N}/\tilde{N}^0$ is finitely generated, a lattice $\tilde{\Lambda}$, a unipotent automorphism $\tau$ with $\tau(\tilde{\Lambda})=\tilde{\Lambda}$, a smooth surjective homomorphism $\eta:\tilde{N}\rightarrow N$, mapping $\eta(\tilde{\Lambda})\subseteq \Lambda$ and some $c\in \tilde{N}$ such that $g(n)=\eta(\tau^n(c))$. 
\end{lemma}
\begin{proof}
    \cite[Proposition 3.14]{Leibmanpointwise}.
\end{proof}

The second result is an equidistribution result in the spirit of Ratner's theorems for nilmanifolds.
\begin{lemma}[Orbit closures of nilrotations {\cite{Leibmanpointwise}}]\label{lem:leibman-orbit}
Let $N$ be a nilpotent Lie group of step $r$, let $\Lambda \leq N$ be a lattice, and let $X=N/\Lambda$. For every $s\in N$ and $x=n\Lambda\in X$, write  $$\mathcal{O}(x) = \overline{\{s^n x : n\in \Z\}}$$ for the orbit closure of $x$ under the action of $s$ and let $H_x:=\{h\in N : h\mathcal{O}(x)=\mathcal{O}(x)\},$ denote the subgroup of $N$ of all the translations which preserve the orbit closure. Then the following hold:
\begin{itemize}
    \item[(1)] $H_x$ is closed in $N$, contains $s$, and satisfies that $\mathcal{O}(x) = H_xx$. The stabilizer $\Lambda_x:=H_x\cap n\Lambda  n^{-1} $ of $x$ in $H_x$ is discrete and co-compact in $H_x$.
    \item[(2)] $\mathcal{O}(x)$ is a compact real-analytic submanifold of $X$. The map
    \begin{align*}\imath&:H_x/\Lambda_x\rightarrow \mathcal{O}(x)\\ \imath&(h\Lambda_x)=hx
    \end{align*} is an analytic diffeomorphism of $H_x/\Lambda_x$ onto $\mathcal{O}(x)$. 
\end{itemize}
\end{lemma}
\begin{proof}
First, we may assume without loss of generality that $x=e\Lambda$ is the base point. Indeed, if $x=n\Lambda$, the map $h\Lambda\mapsto hn^{-1}(n\Lambda n^{-1})$ is an $N$-equivariant analytic diffeomorphism from $N/\Lambda$ onto $N/n\Lambda n^{-1}$ sending $x$ to the base point, so we may replace $\Lambda$ with $n\Lambda n^{-1}$. In particular, $\Lambda_x=H_x\cap\Lambda$.

To prove $(1)$, observe first that $H_x$ is a subgroup by definition, and it is closed since
$$
H_x=\bigcap_{y\in\mathcal{O}(x)}\{h\in N:\ hy\in\mathcal{O}(x)\text{ and }h^{-1}y\in\mathcal{O}(x)\},
$$
and each set in the intersection is closed because $\mathcal{O}(x)$ is closed and the action is continuous. The fact that $s\in H_x$ is immediate. By \cite[Theorem~2.21]{Leibmanpointwise}, there is a closed subgroup $E\subseteq N$ with $\mathcal{O}(x)=Ex$. Clearly $E\subseteq H_x$, and therefore $\mathcal{O}(x)=H_xx$. Since $H_xx$ is closed, the discrete subgroup $\Lambda_x$ is cocompact in $H_x$, and $\iota$ is a homeomorphism by \cite[Section~2.2]{Leibmanpointwise}. The equivariance of $\iota$ is immediate.

To prove $(2)$, recall that by the closed subgroup theorem (see e.g.\ \cite[Ch.~II, Thm.~2.3]{Hel78}), $H_x$ is an embedded real analytic submanifold of $N$, and hence so is $H_x\lambda$ for every $\lambda\in\Lambda$. Since $\Lambda$ and $\Lambda_x$ are discrete, the projections $\pi\colon N\to X$ and $\pi_x\colon H_x\to H_x/\Lambda_x$ are local analytic diffeomorphisms.

We first show that $H_x\Lambda=\pi^{-1}(\mathcal{O}(x))$ is an embedded analytic submanifold of $N$. Since $H_x/\Lambda_x$ is compact, $H_x=K\Lambda_x$ for some compact $K\subseteq H_x$. Let $p\in H_x\lambda_0$ and let $C$ be a compact neighborhood of $p$. If $h\lambda\in C$ with $h=k\lambda'$, $k\in K$, $\lambda'\in\Lambda_x$, then $\lambda'\lambda\in\Lambda\cap K^{-1}C$, which is a finite set since $\Lambda$ is discrete. As $H_x\lambda=H_x\lambda'\lambda$, only finitely many cosets $H_x\lambda$ meet $C$. Distinct cosets are disjoint and closed, so $p$ has an open neighborhood $O$ with $O\cap H_x\Lambda=O\cap H_x\lambda_0$. Hence $H_x\Lambda$ is an embedded analytic submanifold of $N$.

Now let $p\in H_x\Lambda$, and let $V\ni p$ be open such that $\pi|_V$ is an analytic diffeomorphism onto the open set $\pi(V)$. Since $\pi^{-1}(\mathcal{O}(x))=H_x\Lambda$, we have $\mathcal{O}(x)\cap\pi(V)=\pi(V\cap H_x\Lambda)$, so $\mathcal{O}(x)$ is an embedded analytic submanifold of $X$ near $\pi(p)$. If moreover $p\in H_x$ and $V\subseteq O$, then $V\cap H_x\Lambda=V\cap H_x$, so $\pi$ restricts to an analytic diffeomorphism from $V\cap H_x$ onto the open subset $\mathcal{O}(x)\cap\pi(V)$ of $\mathcal{O}(x)$. Hence $\pi|_{H_x}\colon H_x\to\mathcal{O}(x)$ is a local analytic diffeomorphism. Since $\iota\circ\pi_x=\pi|_{H_x}$ and $\pi_x$ is a local analytic diffeomorphism, the bijection $\iota$ is a local analytic diffeomorphism and hence an analytic diffeomorphism.
\end{proof}
As a corollary, we get the following lemma. This lemma introduces an auxiliary nilsystem $\Omega$ which will be used throughout the paper.
\begin{lemma}\label{auxiliary:lemma}
    Let $X=N/\Lambda$ be a connected nilmanifold and $g:\Z\rightarrow N$ a polynomial. There exist an ergodic nilsystem $(\Omega,\mu_\Omega,S)$, a point $\omega_0$ with distinct iterates $\omega_n=S^n\omega_0$ and an analytic map $\eta:\Omega\rightarrow X$ such that \begin{equation}\label{etag}\eta(S^n\omega_0) = g(n)\Lambda
    \end{equation} for all $n\in \Z$. When  $g$ is linear, i.e., $g(n)=a^n$ for some $a\in N$, $\Omega$ can be taken of the same step as $N$.
\end{lemma}
\begin{proof}
Write $g(n)=a_1^{p_1(n)}\cdots a_s^{p_s(n)}$. 
We may choose a presentation so that $N=\langle N^0,a_1,\dots,a_s\rangle$. Since $X$ is connected, $N=N^0\Lambda$. Let $\eta$, $\tilde{N}$, $\tilde{\Lambda}$, $c$ and $\tau$ as in Lemma~\ref{Leibman:realization} and put: 
    $$N' = \tilde{N}\rtimes_\tau \Z,\quad \Lambda'=\tilde{\Lambda}\rtimes_\tau \Z,\quad s=(e,1).$$

    Clearly, $\Lambda'\leq N'$ is discrete, and the quotient map $N'/\Lambda'\rightarrow \tilde{N}/\tilde{\Lambda}$ is an analytic isomorphism. In particular, $\Lambda'$ is co-compact. 
    
    Since $\tau(\tilde\Lambda)=\tilde\Lambda$ and $\eta(\tilde\Lambda)\subseteq\Lambda$, the map
\begin{align*}\bar\eta&:N'/\Lambda'\rightarrow X\\ \bar\eta\bigl(&(x,m)\Lambda'\bigr)=\eta(x)\Lambda,\end{align*}
is well defined, and it is analytic by Proposition~\ref{lieisanalytic:prop}(i). As $s^n(c,0)=(\tau^n(c),n)=(\tau^n(c),0)\cdot(e,n)$ with $(e,n)\in\Lambda'$, Lemma~\ref{Leibman:realization} gives
$$\bar\eta\bigl(s^n(c,0)\Lambda'\bigr)=\eta(\tau^n(c))\Lambda=g(n)\Lambda.$$

    We would like to take $\Omega=N'/\Lambda'$, and $\omega_0\coloneqq(c,0)\Lambda'$, but a priori we have no way to guarantee that the iterates $s^n \omega_0$ are distinct. Therefore, we make a small modification: let $N'/\Lambda'\times \T$, where $\T=\R/\Z$, equipped with the action $(x,t)\mapsto (sx,t+\theta)$ for some irrational $\theta$, and let $\Omega$ be the orbit closure of $\omega_0\coloneqq((c,0)\Lambda',0)$ under $S=(s,\theta)$. By the second Lemma of Leibman (Lemma~\ref{lem:leibman-orbit}), $\Omega$ is a nilmanifold. Define $\eta:\Omega\rightarrow X$ by $\eta(y,t)=\bar\eta(y)$. It is analytic, as the restriction of an analytic map to the embedded analytic submanifold $\Omega$ (Lemma~\ref{lem:leibman-orbit}), and it satisfies \eqref{etag}. The circle coordinate ensures that $S^n\omega_0$ are pairwise distinct. When $g(n)=a^n$, we can take $\Omega$ the orbit closure of $(\Lambda,0)$ under $S=(a,\theta)$, and the same argument holds. 
\end{proof}
\subsection{Quantitative equidistribution}
Next, we survey some quantitative equidistribution results of Green and Tao \cite{gt1} and derive some corollaries.\\

Throughout this section $L$ is a connected, simply connected nilpotent Lie group with Lie algebra $\lalg$, $\Delta$ is a lattice, and $M=L/\Delta$ is a nilmanifold. For a connected and rational subgroup $J\leq L$, we denote by $\mu_J$ the normalized Haar measure on the subnilmanifold $J\Delta/\Delta$. Write $e(t)=\exp(2\pi i t)$. The following definition involves some well known facts (see e.g., \cite{hk-book}).
\begin{definition}[Horizontal characters]
 A \emph{horizontal character} is a continuous homomorphism $\chi:L\rightarrow \T$ which vanishes on $\Delta$. It therefore factors through $L/L_2\Delta$. The factor $L/L_2\Delta$ is called a \emph{horizontal torus} and is isomorphic to $\T^{d_1}$ where $d_1=\dim L/L_2 = \dim L-\dim L_2$. Fix some coordinates for this torus, and write $|\chi|_2$ for the Euclidean norm of $\chi$ in these coordinates.  We write $d_\chi:\lalg\rightarrow \R$ for the derivative of $\chi$ (so $\chi\circ \exp = d_\chi \pmod \Z $).
\end{definition}
The following result is a direct corollary of Green and Tao's quantitative version of Leibman's equidistribution theorem. Following their notations, we fix a strong Mal'cev basis $\mathcal{X}$ adapted to filtration on $L$ for $M$ \cite[Definition 2.1]{gt1} which is $Q_0$-rational \cite[Definition 2.4]{gt1} for some $Q_0\geq 2$, and let $d_\mathcal{X}$ the smooth metric on $M$ associated with the basis \cite[Definition 2.2]{gt1}. Let $d_M$ be a smooth Riemannian metric on $M$, we write 
$$\|F\|_{\mathrm{Lip}} = \|F\|_\infty + \sup_{x\not=y} \frac{|F(x)-F(y)|}{d_M(x,y)},\quad \|F\|_{\mathrm{Lip},\mathcal{X}} = \|F\|_\infty + \sup_{x\not=y} \frac{|F(x)-F(y)|}{d_{\mathcal{X}}(x,y)}.$$
\begin{lemma}\label{thm:GT}
    For every $0<\varepsilon<\frac{1}{2}$, there is a constant $B=B(M,\varepsilon)\geq 1$, depending only on $\varepsilon$ and on the fixed variables $(M,\mathcal{X},d_M)$, with the following property: for all $a\in L$ and $N\geq 2$, failure of $\varepsilon$-equidistribution, i.e.,
$$\left|\frac{1}{N}\sum_{n=1}^N F(a^n \Delta) - \int_M F\, d\mu_L \right| > \varepsilon \|F\|_{\mathrm{Lip}}$$ for a Lipschitz function $F:M\rightarrow \C$, implies the existence of a non-zero horizontal character $\chi$ with $|\chi|_2\leq B$ and $\|\chi(a)\|_{\T}\leq B/N$, where $\|x\|_{\T} = \min \{|\{x\}|,|1-\{x\}|\}$ is the usual metric on $\T$.
\end{lemma}
\begin{proof}
Fix some $K$ so that $\|F\|_{\mathrm{Lip},\mathcal{X}}\leq K\|F\|_{\mathrm{Lip}}$ for all Lipschitz $F:M\rightarrow \mathbb{C}$ (cf., \cite[footnote 4]{gt1}). Applying \cite[Theorem 1.16]{gt1}, in dimension $m$ and degree $r$, there is some $C=C(m,r)\geq 1$ with the following property: for every $\delta\in(0,\frac{1}{2})$, and polynomial $g:\Z\rightarrow L$, if $\mathcal{X}$ is $\frac{1}{\delta}$-rational and $(g(n)\Delta)_{n\in[N]}$ is not $\delta$-equidistributed, then there is a non-trivial horizontal character $\chi$ with $|\chi|\leq C\cdot \delta^{-C}$ and 
$$\|\chi(g(n)) - \chi(g(n-1))\|_{\T} \leq C\delta^{-C}/N$$ for all $n\in [N]$. Set $\delta\coloneqq \min\{\varepsilon/K,1/(2Q_0)\}$ and $B:=C\delta^{-C}$. Then $\delta\in (0,\frac{1}{2})$ and $B\geq 1$ depends only on $m,r,K,Q_0$ and $\varepsilon$. Applying this with $g(n)=a^n$, we get that $|\chi|\leq B$ and $\|\chi(a^n)-\chi(a^{n-1})\|_{\T} = \|\chi(a)\|_{\T}\leq B/N$ for all $n\in [N]$. This completes the proof.
\end{proof}
Next, we wish to study connected subgroups $J\leq L$. Set $\Delta_J:=\Delta\cap J$. We record a simple lemma. \begin{lemma}[The existence of ergodic translations.]\label{lem:generic-a}
Let $J\le L$ be a closed connected rational subgroup. Let $\mu_J$ be the Haar probability measure of the nilmanifold $J/\Delta_J$. Then there is $a\in J$ whose action on $J/\Delta_J$ is ergodic. In fact, this holds true for Haar-almost every $a\in J$.
\end{lemma}

\begin{proof}
If $J=\{e\}$, take $a=e$. Suppose that $\dim J\geq 1$, and let $\mathfrak{j}$ denote the Lie algebra of $J$. Let $\chi:J\rightarrow \mathbb{T}$ be a horizontal character. If $\chi\not=0$, then since $\exp:\mathfrak{j}\rightarrow J$ is onto, $\chi(\exp(V)) \not=0$ for some $V$. Therefore, the derivative $d_\chi(V)$ is not an integer and in particular, $d_\chi$ is not trivial. Let $m\in \mathbb{Z}$, the level set
$$H_{\chi,m} = \{V\in \mathfrak{j} : d_\chi(V)=m\}$$ is an affine hyperplane in $\mathfrak{j}\cong \mathbb{R}^{\dim J}$. In particular, it is a null set with respect to the Lebesgue measure. Since there are at most countably many horizontal characters, the union of $H_{\chi,m}$ over all $\chi,m$ is also Lebesgue-null. Hence, we can choose $V$ outside this union (and almost every $V$ is outside this union). Take $a=\exp(V)\in J$. From the argument above, $\chi(a)\not =0$ for every non-zero horizontal character on $J$. Since the characters generate the $L^2$-space of the horizontal torus, we conclude that the induced action of $a$ on $J/J_2\Delta_J$ is ergodic. By Green's theorem the action of $a$ on $J/\Delta_J$ is also ergodic.
\end{proof}
Suppose that $J\leq L$ is a connected subgroup. Roughly speaking, we can measure how much $L$ is larger than $J$ by studying the horizontal characters that vanish on $J$. The following lemma is a useful corollary of the previous theorem which makes this precise.
\begin{lemma}\label{vanishorequi}
  For every $\varepsilon\in(0,\frac{1}{2})$, there is a finite set $\mathcal{S}_{\varepsilon}$ of non-zero horizontal characters of $M$ such that, for each connected rational $J\leq L$ one of the following holds:
  \begin{itemize}
      \item (Vanishing) Some $\chi\in \mathcal{S}_{\varepsilon}$ vanishes on $J$, or
      \item (approximating) 
      \begin{equation}\label{apprx}\left| \int_M F\, d\mu_J - \int_M F d\mu_L\right| \leq 2\varepsilon \|F\|_{\mathrm{Lip}}
      \end{equation} for every Lipschitz function $F$.
  \end{itemize}
\end{lemma}
\begin{proof}
First choose an ergodic translation $a\in J$, which exists by the previous lemma. Since $J/\Delta_J$ is then uniquely ergodic, the pointwise ergodic theorem applies everywhere for continuous functions. In particular,
\begin{equation}\label{equiJ}\lim_{N\rightarrow\infty} \frac{1}{N}\sum_{n=1}^N F(a^n\Delta_J) = \int_M F\,d\mu_J
\end{equation}
for all continuous $F:J/\Delta_J\rightarrow \C$ (and therefore also for all continuous $F$ on $M$).

Apply Lemma~\ref{thm:GT} for this $a$ on $M$. We get a constant $B_M=B(M,\varepsilon)$ such that if the following equidistribution in $M$ fails,
$$\left|\frac{1}{N}\sum_{n=1}^N F(a^n \Delta) - \int_M F\, d\mu_L \right| > \varepsilon \|F\|_{\mathrm{Lip}}$$ then there exists a horizontal character on $L/\Delta$ such that $\|\chi(a)\|_{\T}\leq B_M/N$ and $|\chi|_2\le B_M$.
Let $S_{\varepsilon}$ be the set of all non-trivial horizontal characters with $|\chi|_2\leq B_M$. Then it is a finite set. If no $\chi$ in $S_\varepsilon$ vanishes on $J$, then by ergodicity, for $N$ sufficiently large $\|\chi(a)\|_{\T}> B_M/N$. In this case for all $N$ sufficiently large, we must have
$$\left|\frac{1}{N} \sum_{n=1}^N F(a^n\Delta_J) - \int_M F\, d\mu_L\right| \leq \varepsilon\|F\|_{\mathrm{Lip}}$$ for all Lipschitz functions. Combining with \eqref{equiJ} we get the desired result.
\end{proof}
The following proposition characterize when a sequence of connected subgroups $J_j$ approximate $L$.
\begin{proposition}\label{approximation:prop}
    Let $J_n\leq L$ be a sequence of connected closed rational subgroups. Then, $\mu_{J_n}\rightarrow \mu_L$ weakly if and only if every non-zero horizontal $\chi$ is eventually non-zero on $J_n$.
\end{proposition}
\begin{proof}
This is a corollary of the previous lemma. First observe that
    \begin{equation}\label{J-orth}
\int_M e(\chi(x)) d\mu_J(x) = \begin{cases}  1 & \chi \text{ vanishes on } J\\
0 & \text{otherwise.}\end{cases}
    \end{equation}
To see this, choose $a\in J$ which acts ergodically. Since the measure is invariant to translations by $a$ we have $$\int_M e(\chi(x)) d\mu_J(x) = e(\chi(a))\cdot\int_M e(\chi(x)) d\mu_J(x)$$ and we proved before that for an ergodic $a$ $\chi(a)=0$ iff $\chi|_J=0$, as required.

Therefore, if $\mu_{J_n}\rightarrow \mu_L$ weakly, then $\int_M e(\chi(x))d\mu_{J_n}\rightarrow\int_M e(\chi(x))d\mu_{L}$, and therefore every non-zero horizontal character $\chi$ on $L$ is eventually non-zero on $J_n$. 

Conversely, suppose that every non-zero horizontal character is eventually non-trivial on $J_n$, we wish to prove that $\mu_{J_n}\rightarrow \mu_L$ weakly. Fix a Lipschitz function $F$, by the density of the Lipschitz functions in the space of continuous functions it suffices to show that
$$\lim_{n\rightarrow\infty} \int F\, d\mu_{J_n} = \int_M F\,d\mu_L.$$ 

Let $\varepsilon\in(0,\frac{1}{2})$, it suffices to find some $N$ so that for all $n>N$ we have
$$ \left|\int F\, d\mu_{J_n} - \int_M F\,d\mu_L\right|\leq 2\varepsilon\|F\|_{\mathrm{Lip}}.$$ 

To obtain this we apply the previous lemma. Indeed, we get a finite set $\mathcal{S}$ of non-zero horizontal characters. Since $\mathcal{S}$ is finite, our assumption implies that for $N$ sufficiently large, every $\chi\in \mathcal{S}$ is non-zero on $J_n$ for all $n>N$. Thus, \eqref{apprx} gives the desired result.
\end{proof}

\section{Proof of Theorem~\ref{leibman:thm}}
Suppose we wish to prove that a sequence $a(n)$ is in $\GenNz_{s}$, and suppose we have already identified the underlying nilsystem, denoted by $M=(N/\Lambda,R_\alpha)$. A natural approach is to define $F:M\rightarrow \C$ by 
\begin{equation}\label{Fdef}
   F(x)= \begin{cases}
        a(n) & x=\alpha^n\Lambda\\
        0 & \text{otherwise.}
    \end{cases}.
\end{equation}
$F$ is well defined when $\alpha^n\Lambda\not=\alpha^m \Lambda$ unless $n=m$. In the next lemma we establish precisely when this choice of $F$ is Riemann integrable.
\begin{definition}[Controlled sequences]\label{def:controlled}
Let $(\Omega,S)$ be an ergodic nilsystem and $\omega_0\in\Omega$. A bounded sequence $a:\Z\rightarrow\C$ is \emph{controlled by $(\Omega,\omega_0)$} if there is a $\mu_\Omega$-null set $E\subseteq\Omega$ such that $a(n_j)\rightarrow0$ whenever $|n_j|\rightarrow\infty$ and $S^{n_j}\omega_0\rightarrow\omega\notin E$.
\end{definition}
\begin{lemma}[Sequences in $\GenNz_s$ are dynamically controlled]\label{control:lem}
Let $s\geq 1$. A bounded sequence is in $\GenNz_{s}$ if and only if it is controlled by $(X,x)$ for some ergodic $s$-step nilsystem $X$ and some $x\in X$.
\end{lemma}
The key takeaway is the fact that this criterion is a \emph{local property}. Namely, once the sequence and the limit $x_0$ is fixed the behavior depends on a small neighborhood around $x_0$. This will later allow us to apply analytic tools to resolve Theorem~\ref{leibman:thm}. We now prove this lemma.
\begin{proof}
For the first direction, suppose that $a\in \GenNz_s$. By definition, there exists an ergodic $s$-step nilsystem $X=(N/\Lambda,R_\alpha)$ and a Riemann-integrable function $F:X\rightarrow \C$ vanishing outside of a Haar null set $E'$ and such that $a(n)=F(\alpha^n x\Lambda)$. Replacing $\Lambda$ by $x\Lambda x^{-1}$ and modifying $F$ accordingly, we may assume without loss of generality that $x=e$. Let $E=E'\cup \mathrm{Disc}(F)$ where $\mathrm{Disc}(F)$ denotes the set of all discontinuities of $F$. If $|n_j|\rightarrow \infty$ and $\alpha^{n_j}\Lambda\rightarrow x_0\not\in E$, then by the continuity of $F$ in $x_0$ and the fact that $F(x_0)=0$, we deduce that $\lim_{j\rightarrow\infty} a(n_j) = F(x_0)=0.$ This completes the proof of the first direction.

For the second direction, suppose that $a$ is controlled by $(X,x)$, where $X=(N/\Lambda,R_\alpha)$, with null set $E$. Replacing $\Lambda$ by a conjugate, we may assume that $x=e\Lambda$. Write $\mathcal{O}(e\Lambda)$ for the orbit of $e\Lambda$. If $\mathcal{O}(e\Lambda)$ is finite, ergodicity implies that the nilmanifold is finite, but then the set $E$ from the assumption must be empty, and therefore $\lim_{|n|\rightarrow\infty} a(n)=0$. In this case, the assumption of the lemma holds for every ergodic $s$-step nilsystem with $E=\emptyset$, for instance for an irrational rotation on $\T$, whose orbits are infinite. Hence, replacing $X$ by such a system, it suffices to treat the case where $\mathcal{O}(e\Lambda)$ is infinite.

It is left to solve for the case where $\mathcal{O}(e\Lambda)$ is infinite. In that case ergodicity implies it is Haar null and the orbit sequence $\alpha^n \Lambda$ are pairwise distinct. Let $F$ be as in \eqref{Fdef}, we see that $F$ vanishes almost everywhere. Let $E$ be as in the assumption, we prove that $\mathrm{Disc}(F)\subseteq E\cup\mathcal{O}(e\Lambda)$ (hence completing the proof). Let $x_0\not\in E\cup \mathcal{O}(e\Lambda)$, since $F(x_0)=0$, it suffices to show that whenever $x_j\rightarrow x_0$ we have $F(x_j)\rightarrow 0$. Since $F$ vanishes outside the orbit, we may exclude every $x_j$ not from the orbit. Hence, $x_j = \alpha^{n_j}\Lambda$ without loss of generality. Since $x_0\not\in \mathcal{O}(e\Lambda)$, we must have $|n_j|\rightarrow \infty$. Hence, by the assumption 
$$\lim_{j\rightarrow\infty} F(x_j) =\lim_{j\rightarrow\infty} a(n_j)=0,$$ as required.
\end{proof}

In both Theorem~\ref{leibman:thm} and Theorem~\ref{FK:thm}, we need to identify an error term and show that it equals to a generalized nilsequence which vanishes outside of a null set. 

We start with Theorem~\ref{leibman:thm}, the proof of Theorem~\ref{FK:thm} will be left till later. Let $\XX=N/\Lambda$ be a nilmanifold $\YY=H\Lambda/\Lambda$ with $H\leq N$ closed connected and rational and $g:\Z\rightarrow N$ be a polynomial sequence. Let $F\in C(X)$ be arbitrary. Leibman's equidistribution results give rise to a candidate: We say that $F$ is an \emph{error term for $(Y,g)$} if 
\begin{equation}\label{weakmixing}\lim_{N\rightarrow\infty}\frac{1}{N}\sum_{n=1}^N\left|\int_{g(n)Y} F\,d\mu_{g(n)Y}\right|=0.
\end{equation}
\begin{lemma}[Structured and error terms for Theorem~\ref{leibman:thm}]\label{str-err:lem} Let $\XX,\YY,\ZZ$ and $g$ be as in Theorem~\ref{leibman:thm}, and write $\ZZ=K\Lambda/\Lambda$ where $K$ is the normal closure of $H$ in $N$. Then $F=F_s+F_e$ is a sum of two continuous functions such that:
    \begin{itemize}
        \item (Structured component) $$F_s(x\Lambda) = \int_{xZ}F\,d\mu_{xZ},$$ where $\mu_{xZ}$ is the Haar measure on $xZ$, is well defined and independent of the choice of $x$. Moreover, it satisfies that $F_s(g(n)\Lambda) = \int_{g(n)Z} F\,d\mu_{g(n)Z}$ (thus, $b_{F_s}(n) =0$).
        \item (error component) $$F_e=F-F_s$$ is an error term for $(Y,g)$.
    \end{itemize}
\end{lemma}
\begin{proof}
Clearly $F=F_s+F_e$. For the structured component, observe that $F_s$ is well defined: if $\lambda\in\Lambda$, then $x\lambda Z=xZ$ since $K$ is normal, and $\mu_{x\lambda Z}=\mu_{xZ}$ since conjugation by $\lambda$ is an automorphism of $K$ preserving $K\cap\Lambda$. Moreover, since $H\subseteq K$, $F_s$ is constant on $g(n)Y$ and $g(n)Z$, on both its value is $F_s(g(n)\Lambda)=\int_{g(n)Z} F\,d\mu_{g(n)Z}$. Let $F_e=F-F_s$, then, 
$$\lim_{N\rightarrow\infty}\frac{1}{N} \sum_{n=1}^N \left|\int_{g(n)Y}F_e\,d\mu_{g(n)Y} \right| = \lim_{N\rightarrow\infty}\frac{1}{N}\sum_{n=1}^N |b_F(n)|=0,$$ where the last equality was established by Leibman \cite[Proposition 3.1]{Leibman1}.
\end{proof}
The original, now refuted, conjecture of Leibman (Conjecture~\ref{Leibman:conj}) predicted that the vanishing of the Ces\'aro average alone (in absolute value) implies that $\int_{g(n)Y} F\,d\mu_{g(n)Y}$ vanishes as $n\rightarrow\infty$. This is false. However, the dynamical system $\Omega$ constructed in Lemma~\ref{auxiliary:lemma} will give us a different (weaker) dynamical control on this quantity. 
\begin{theorem}[Dynamical control]\label{control:thm}
    Let $\XX=N/\Lambda$ be a connected nilmanifold, let $\YY=H\Lambda/\Lambda$ where $H\leq N$ is closed connected and rational and let $g:\Z\rightarrow N$ be a polynomial sequence. Let $(\Omega,S)$ be an ergodic nilsystem and suppose that $\eta:\Omega\rightarrow X$ is analytic such that 
    $$\eta(S^n \omega_0) = g(n)\Lambda,$$ for some $\omega_0\in \Omega$ and all $n\in \Z$.
    Then there exists a Borel $\mu_{\Omega}$-null subset $E\subseteq \Omega$ such that for every error term $F$ for $(Y,g)$:  
    if $|n_j|\rightarrow \infty$ and $S^{n_j}\omega_0\rightarrow \omega\not\in E$, then
    $$\lim_{j\rightarrow\infty} \int_{g(n_j)Y} F\,d\mu_{g(n_j)Y}=0.$$
\end{theorem}
Informally, the vanishing is determined by the paths of the orbit of $\omega_0$ which end up outside of a null set.

For the purpose of the proof of Theorem~\ref{control:thm} we rely on the key fact that the zero set of a non-zero real analytic function has measure zero. In its simplest form it was established in \cite{Mityagin}, and the general case is established in Appendix~\ref{zero:section}. 
\begin{proof}[Proof of Theorem~\ref{control:thm}]
As before, write $\XX=N/\Lambda$ and let $N^0$ be the connected component of the identity in $N$.

Our first goal is to represent $\XX$ as a quotient of a connected simply connected Lie group. Since $\XX$ is connected, we have $N/\Lambda\cong N^0/\Lambda^0$ as nilmanifolds (analytically), where $\Lambda^0=\Lambda\cap N^0$. Let $p:L\rightarrow N^0$ be the universal cover map and let $\Delta=p^{-1}(\Lambda^0)$. Then $L$ is connected, simply connected, and nilpotent and $\Delta$ is discrete. Moreover, $L/\Delta\cong X$ as nilmanifolds.

Next, we study how this identification is compatible with $\YY$. Recall that $\YY=H\Lambda/\Lambda$ for some connected $H$, hence $H\subseteq N^0$. Let $\widetilde{H}$ be the connected component of the identity of $p^{-1}(H)$, and note that $p(\widetilde{H})=H$. The covering map $p$ therefore induces an isomorphism of nilmanifolds
$$\widetilde{H}/(\widetilde{H}\cap \Delta)\cong H/(H\cap \Lambda).$$ In particular, $\widetilde{H}$ is a rational subgroup of $L$.

Finally, since $N^0$ is the maximal connected component,  the conjugation map $n\mapsto ana^{-1}$ preserves $N^0$. We denote this map by $c_a:N^0\rightarrow N^0$. Since $L$ is simply connected, the map $c_a\circ p:L\rightarrow N^0$ admits a unique lift taking $e$ to itself, denoted by $\widehat{c_a}:L\rightarrow L$. The same uniqueness also shows that $\widehat{c_a}$ is an analytic automorphism and the following holds:
$$p\circ \widehat{c_a} = c_a\circ p,\text{  and  } \widehat{c_{ab}} = \widehat{c_a}\circ \widehat{c_b}.$$ In particular, if $\lambda\in \Lambda$, then $c_\lambda$ preserves $\Lambda^0$ and therefore $\widehat{c_\lambda}$ preserves $\Delta$. Thus, $\widehat{c_\lambda}(\widetilde{H})$ is also rational in $L$. 

Now that we have established and studied the new presentation, we wish to study the polynomial $g$. We adapt the notation of the discrete derivative $\partial_kg(n) = g(n+k)\cdot g(n)^{-1}$ and write $$g(n)=a_1^{p_1(n)}\cdots a_s^{p_s(n)}$$ for some $a_1,\dots,a_s\in N$ and polynomials $p_1,\dots,p_s$  which we assume are all of degree $\leq d$ for some $d\geq 1$. Let $r$ be such that $N$ is $r$-step nilpotent and let $D=rd$. 

Even though $p_i(n)$ is a polynomial of degree $d$, the degree of each component $a_i^{p_i(n)}$ also depends on the place of $a_i$ in the lower central series of $N$, but it is always of degree at most $rd$. It is convenient to follow the notations of Leibman \cite{leibmanpolyseq}: recall that a function $f:\Z\rightarrow G$ is a polynomial with respect to a filtration $G_\bullet=(G_i)_{i\geq 0}$ if $\partial_{k_1}\dots \partial_{k_j} f \in G_{j}$ for all $k_1,\dots, k_j\in \Z$. Following Green and Tao, we define the filtration $N_\bullet=(N_i)_{i\geq 0}$, where $N_0=N$ and $N_j=\gamma_{\lceil j/d\rceil}(N)$ for $j\geq 1$, where $\gamma_i(N)$ is the $i$-th group in the lower central series for $N$. Under this filtration, every component $n\mapsto a_i^{p_i(n)}$ is a polynomial for this filtration (cf., \cite[Proof of Corollary 6.8]{gt1}), and Lazard-Leibman theorem (see e.g., \cite[Proposition 6.2]{gt1}) implies that $g$ is then also a polynomial for this filtration, and for every $h\in \widetilde{H}$, so is $z_h(n)\coloneqq g(n)p(h)g(n)^{-1} $.

We may lift $z_h$ by setting $\widetilde{z_h}(n)=\widehat{c_{g(n)}}(h)$. We show that $\partial_{k_1}\cdots \partial_{k_{D+1}}\tilde{z_h}\equiv e$ for all $k_1,\dots k_{D+1}\in\Z$. To see this, observe that since $\partial_{k_1}\cdots \partial_{k_{D+1}}z_h(n)=e$ we must have $\partial_{k_1}\cdots \partial_{k_{D+1}}\tilde{z_h}(n)\in \ker p$ for all $n$. However, since $\ker p$ is discrete, $\widetilde{H}$ is connected, and the map $h\mapsto \partial_{k_1}\cdots \partial_{k_{D+1}}\widetilde{z_h}(n)$ is continuous and is trivial at $e_{\tilde{H}}$ we conclude,
\begin{equation}\label{Dzh}
    \partial_{k_1}\cdots \partial_{k_{D+1}}\tilde{z_h}\equiv e,
\end{equation} for all $k_1,\dots,k_{D+1}\in \Z$, as required.

Now we take advantage of the fact that $\Omega$ is a real analytic manifold (Lemma~\ref{lem:leibman-orbit} and Proposition~\ref{lieisanalytic:prop}(iii)). Let $\mathcal{U}$ be a finite cover of $\Omega$ by connected simply connected analytic coordinate neighborhoods such that, for each $U\in\mathcal{U}$ the restricted map $\eta|_U$ admits a lift $b:U\rightarrow L$ (Lemma~\ref{analyticlift:lem}). Here the lift is taken through the covering map $L\rightarrow X$. That is, for all $u\in U$ we have
$$p(b(u))\Lambda = \eta(u).$$

Fix a specific pair $(b,U)$. For $n\in \Z$ and $u\in U$ define an automorphism of $L$ by
$$\alpha_{n,u}(h) = b(u)^{-1} \widehat{c_{g(n)}}(h)b(u)$$ and let $K_{n,u}$ denote the image of $\widetilde{H}$ under this automorphism. Conjugating \eqref{Dzh} by $b(u)$ then gives that
\begin{equation}\label{Dah}
    \partial_{k_1}\dots\partial_{k_{D+1}} \left(n\mapsto \alpha_{n,u}(h)\right)=e,
\end{equation}
for all $k_1,\dots, k_{D+1}\in \Z$, $u\in U$ and $h\in \widetilde{H}$.

Let $J_U$ be a rational connected subgroup of $L$ of minimal dimension which contains every $K_{n,u}$ for all $n\in \Z$ and $u\in U$. Set $\Delta_U=J_U\cap \Delta$ and $M_U=J_U/\Delta_U$. Since $J_U\leq L$ is connected, it is also simply connected. Therefore, any horizontal character $\chi$ on $M_U$ admit a unique real-analytic lift, denoted by $\eta_\chi:J_U\rightarrow \mathbb{R}$ with $\eta_\chi(e)=0$ (Lemma~\ref{analyticlift:lem}). In fact, this lift is a homomorphism because $(x,y)\mapsto \eta_\chi(xy)-\eta_\chi(x)-\eta_\chi(y)$ is an integer valued map on a connected space $J_U\times J_U$ which takes the trivial value on $(e,e)$.

Now set $$Q_{\chi,h}(n,u) = \eta_\chi\bigr(\alpha_{n,u}(h)\bigl).$$ Since $\eta_\chi$ is a homomorphism, \eqref{Dah} implies that all $D+1$ derivatives of $Q_{\chi,h}$ vanish. Since it is $\R$-valued it is a genuine polynomial of degree at most $D$. Hence, we have
\begin{equation}\label{Q-rep}
    Q_{\chi,h}(n,u) = \sum_{k=0}^D \binom{n}{k} \partial_1^kQ_{\chi,h}(0,u)
\end{equation}
by the discrete Taylor series (here $\partial_1 Q(n,u) = Q(n+1,u)-Q(n,u)$). Write $q_{\chi,h,k}(u) \coloneqq \partial_1^k Q_{\chi,h}(0,u)$ for the coefficients. From the construction all the coefficients are real analytic, as a finite coefficient of a real analytic functions of $u$. Moreover, suppose that $\chi\not= 0$. We claim that there exists some $h_\chi\in \widetilde{H}$ such that $Q_{\chi,h_\chi}$ is not identically zero. If not, then $\eta_\chi$ would vanish on $\alpha_{n,u}(\widetilde{H})=K_{n,u}$ for all $n,u$, i.e., every $K_{n,u}$ would be contained in $\ker\eta_\chi$. Since $\eta_\chi\circ\exp$ is a linear functional on the Lie algebra of $J_U$ which is integer-valued on $\log\Delta_U$ (as $\chi$ vanishes on $\Delta_U$), the group $\ker\eta_\chi$ is a connected rational subgroup of $J_U$, and it is of lower dimension since $\chi\neq 0$. This contradicts the minimality of $J_U$. Henceforth, we fix such choice $\chi\mapsto h_\chi$ whenever $\chi\not=0$ and write
$$d_\chi = \max \{k : q_{\chi,h_\chi,k}\not\equiv 0\text{ on } U\}.$$ Define $$E_U=\bigcup_{\chi\not=0} \{u\in U: q_{\chi,h_\chi,d_\chi}(u)=0\}.$$
Each constituent is the zero set of a non-zero real analytic function on $U$. Lemma~\ref{zeros} and the fact that there are at most countably many horizontal characters, imply that $E_U$ is $\mu_{\Omega}$-null (and if $J_U=\{e\}$, just take $E_U=\emptyset$).

We will now complete the proof of the theorem. Set $E\coloneqq \bigcup_{U\in \mathcal{U}} E_U$, it is null as a finite union of null sets. Suppose that $|n_j|\rightarrow \infty$ and $S^{n_j}\omega_0\rightarrow \omega\not \in E$ and let $U\in\mathcal{U}$ a neighborhood for $\omega$. For every non-zero $\chi$, $Q_{\chi,h_\chi}$ is non-zero. With the definition of $d_\chi$ we see from \eqref{Q-rep} that $\lim_{j\rightarrow\infty}Q_{\chi,h_\chi}(n_j,\omega_{n_j})/n_j^{d_\chi}\not=0$. Namely, $|Q_{\chi,h_\chi}(n_j,\omega_{n_j})|$ is bounded below by a positive constant multiple of $|n_j|^{d_\chi}$ for all sufficiently large $j$. In particular, $\eta_\chi$ is non-zero on $K_{n_j,\omega_{n_j}}$ for all sufficiently large $j$. From which we immediately conclude that $\chi$ itself does not vanish on $K_{n_j,\omega_{n_j}}$.  

Since $p(b(\omega_n))\Lambda=\eta(\omega_n)=g(n)\Lambda$, we have $\lambda_n\coloneqq p(b(\omega_n))^{-1}g(n)\in \Lambda$, i.e., $g(n)=p(b(\omega_n))\lambda_n$. As $\widehat{c_{g(n)}}=\widehat{c_{p(b(\omega_n))}}\circ\widehat{c_{\lambda_n}}$ and $\widehat{c_{p(b)}}(x)=bxb^{-1}$, we get $$K_{n,\omega_n} = b(\omega_n)^{-1}\widehat{c_{g(n)}}(\widetilde{H})b(\omega_n) = \widehat{c_{\lambda_n}}(\widetilde{H})$$ whenever $\omega_n\in U$.

In particular, $K_{n,\omega_n}$ is connected and a rational subgroup of $L$ (with respect to $\Delta)$. It is also rational with respect to $\Delta_U$ since $K_{n,\omega_n}\cap \Delta_U = K_{n,\omega_n}\cap \Delta$. Applying Proposition~\ref{approximation:prop} with $J_n\coloneqq K_{n,\omega_n}$ whenever $\omega_n\in U$, we see that our assumptions imply that \begin{equation}\label{weaklimit}\mu_{K_{n_j,\omega_{n_j}}}\rightarrow \mu_{J_U}
\end{equation}
weakly.

Let $F\in C(X)$ be an error term for $(Y,g)$ and define 
$$\Phi_u(x\Delta) = F(p(b(u)x)\Lambda),$$ for all $u\in U$. This is well defined because $p(\Delta)\subseteq \Lambda$. Since $b$ is continuous, $L/\Delta$ is compact and $F$ is continuous, the map $u\mapsto \Phi_u$ is continuous from $U$ to $C(L/\Delta)$ with the supremum norm, that is, $\|\Phi_{u'}-\Phi_u\|_\infty\rightarrow 0$ as $u'\rightarrow u$. 

For $\omega_n\in U$, and $h\in \widetilde{H}$ we have
$$p(b(\omega_n)\alpha_{n,\omega_n}(h))\Lambda = g(n)p(h)\Lambda,$$ because $g(n)=p(b(\omega_n))\lambda_n$ and $\lambda_n\in \Lambda$. Moreover, $\alpha_{n,\omega_n} = \widehat{c_{\lambda_n}}$ preserves $\Delta$ and induces an isomorphism
$$\tilde{H}/(\tilde{H}\cap \Delta)\rightarrow J_n/(J_n\cap \Delta).$$
Since this preserves the Haar measure, and since $p$ induces the measure-preserving isomorphism $\widetilde{H}/(\widetilde{H}\cap \Delta)\cong H/(H\cap \Lambda)$ from the beginning of the proof, writing $\mu_{\widetilde{H}}$ for the Haar measure on $\widetilde{H}/(\widetilde{H}\cap\Delta)$ we have
\begin{align*}
\int_{L/\Delta} \Phi_{\omega_n}\,d\mu_{J_n}
&= \int_{\widetilde{H}/(\widetilde{H}\cap\Delta)} F\bigl(p(b(\omega_n)\alpha_{n,\omega_n}(h))\Lambda\bigr)\,d\mu_{\widetilde{H}}(h)\\
&= \int_{\widetilde{H}/(\widetilde{H}\cap\Delta)} F\bigl(g(n)p(h)\Lambda\bigr)\,d\mu_{\widetilde{H}}(h)\\
&= \int_{Y} F(g(n)y)\,d\mu_{Y}(y)
= \int_{g(n)Y} F\,d\mu_{g(n)Y},
\end{align*} for all $n$ such that $\omega_n\in U$. Combining this with \eqref{weaklimit}, the continuity of $u\mapsto \Phi_u$ gives
$$\lim_{j\rightarrow\infty}\int_{g(n_j)Y} F\,d\mu_{g(n_j)Y} = \int_{L/\Delta} \Phi_\omega\,d\mu_{J_U}.$$
We note that the same argument applies with any $u\in U\setminus E_U$ in place of $\omega$: whenever $|n_j|\rightarrow\infty$ and $\omega_{n_j}\rightarrow u\in U\setminus E_U$, we have
\begin{equation}\label{limit-u}
\lim_{j\rightarrow\infty}\int_{g(n_j)Y} F\,d\mu_{g(n_j)Y} = \int_{L/\Delta} \Phi_u\,d\mu_{J_U}.
\end{equation}
It therefore suffices to show that $\int_{L/\Delta} \Phi_u\,d\mu_{J_U}=0$ for every $u\in U\setminus E_U$.

However, since $F$ is the error term, the Ces\'aro average \eqref{weakmixing} vanishes by definition. We claim that from this it follows that $\int_{L/\Delta} \Phi_u\,d\mu_{J_U}$ vanishes on $U$. If not, then there exists some $u_0\in U\backslash E_U$ (since $E_U$ is null) such that $a:=\left|\int_{L/\Delta} \Phi_{u_0}\,d\mu_{J_U}\right| >0$. We claim that there is an open neighborhood $V\subseteq U$ of $u_0$ and $N_0$ sufficiently large so that whenever $n>N_0$ and $\omega_n\in V$ we have \begin{equation}\label{contradiction}\left|\int_{g(n)Y} F\,d\mu_{g(n)Y}\right| > a/2.
\end{equation} To see this we again argue by contradiction. Suppose by contradiction that this is not the case. For every $k\in\N$, let $V_k=U\cap B(u_0,\frac{1}{k})$, where the ball is taken with respect to some metric on $\Omega$. The assumption then produces a sequence $n_k\rightarrow\infty$ with $\omega_{n_k}\in V_k$ (hence $\omega_{n_k}\rightarrow u_0$) such that $\left|\int_{g(n_k)Y} F\,d\mu_{g(n_k)Y}\right|\leq a/2$. On the other hand, the first two conditions give (by \eqref{limit-u}) $\int_{g(n_k)Y} F\,d\mu_{g(n_k)Y}\rightarrow\int_{L/\Delta}\Phi_{u_0}\,d\mu_{J_U}$, whose absolute value is $a>a/2$. This is a contradiction.

Notice that our proof uses two (nested) arguments by contradiction. We have just closed the second one. Hence, we have that the assumption (by contradiction) that $a>0$ implies \eqref{contradiction}, and to complete the proof we need to show that this is impossible. Choose a non-zero continuous function $0\leq \varphi\leq 1$ supported on $V$. Since the Haar measure has full support, $\int _{\Omega}\varphi\, d\mu_{\Omega}>0$. Unique ergodicity then gives
$$\lim_{N\rightarrow\infty}\frac{1}{N}\sum_{n=1}^N\left|\int_{g(n)Y} F\,d\mu_{g(n)Y}\right| \geq \frac{a}{2} \lim_{N\rightarrow\infty}\frac{1}{N}\sum_{n=1}^N \varphi(\omega_n)=\frac{a}{2}\int\varphi\,d\mu_{\Omega}>0$$ which contradicts \eqref{weakmixing}. Thus, $\lim_{j\rightarrow\infty} \int_{g(n_j)Y} F \,d\mu_{g(n_j)Y}=\int_{L/\Delta} \Phi_\omega\,d\mu_{J_U} = 0,$
as required.
\end{proof}

We can now complete the proof of Theorem~\ref{leibman:thm} by combining everything together. Let $\XX,\YY,G,H,\Gamma,\pi,g$ and $Z$ be as in the theorem. Let $F$ be a continuous function on $\XX$ and $$b_f(n)=\int_{g(n)Y} F\, d\mu_{g(n)Y} - \int_{g(n)Z} F\,d\mu_{g(n)Z}.$$ By Lemma~\ref{str-err:lem} we have that 
\begin{equation}\label{error}
b_f(n) = \int_{g(n)Y} F_e\, d\mu_{g(n)Y}
\end{equation}
where $F_e$ is the error component of $F$. By Theorem~\ref{control:thm} and Lemma~\ref{control:lem} we deduce that $b_f\in \GenNz_s$, where $s$ is the step of the auxiliary nilsystem $\Omega$. This completes the proof. $\qed$

\section{Proof of Theorem~\ref{FK:thm}}
We first explain why the arguments from before can not be applied directly. First, observe that the nilsystem in Leibman's theorem is connected. If we assume that $\XX$ is connected we can try to write Theorem~\ref{FK:thm} in the language where Theorem~\ref{leibman:thm} applies via the following construction: Let $X=(G/\Gamma,\mu_{G/\Gamma},R_\alpha)$ for some $r$-step nilpotent group $G$ and let 
$$a(n) = \int_X \prod_{i=0}^{\ell} f_i(\alpha^{in}x)\,d\mu_{G/\Gamma}(x).$$
for some $\ell\geq 0$, and bounded functions $f_0,\dots,f_{\ell}$. Let $W=X^{\ell+1}$, $D=\{(x,x,\dots,x)\in W : x\in X\}$ be its diagonal and $F(x_0,\dots,x_{\ell})=\prod_{i=0}^{\ell}f_i(x_i)$ be the tensor function on $W$. Then, for the (linear) polynomial map $g(n)=a^n$, where $a=(e,\alpha,\dots,\alpha^{\ell})\in G^{\ell+1}$ we have
$$a(n) =\int_{g(n)D} F\, d\mu_{g(n)D}.$$
Unfortunately, applying Lemma~\ref{auxiliary:lemma} directly, even though $g$ is linear, we would obtain the nilsystem $\Omega$ whose step is $r$, which might be larger than $\ell$ in general. Instead, we shall rely on the mechanism of \emph{degree lowering} as was introduced by Frantzikinakis and Kuca (see e.g., \cite{FranKuca}). When $r\leq \ell$ and $X$ is connected, however, this direct approach already proves Theorem~\ref{FK:thm}. Indeed, for continuous $f_0,\dots,f_\ell$, Theorem~\ref{leibman:thm}, applied in the closure of $g(\Z)D$ (the closure is connected by \cite[Corollary~6.5]{LeibmanOrbit}), gives $a(n)=\int_{g(n)Z}F\,d\mu_{g(n)Z}+b_F(n)$. The first term is a basic $r$-step nilsequence, and $b_F\in\GenNz_r\subseteq\GenNz_\ell$, since for linear $g$ Lemma~\ref{auxiliary:lemma} gives $\Omega$ of step $r$. For bounded measurable $f_i$, choose continuous $f_i^{(m)}$ with $\|f_i^{(m)}\|_\infty\leq\|f_i\|_\infty$ and $f_i^{(m)}\rightarrow f_i$ in $L^2(\mu_{G/\Gamma})$. Let $F^{(m)}=\prod_i f_i^{(m)}(x_i)$, and write $a^{(m)}=s^{(m)}+e^{(m)}$ for the decomposition above, with $s^{(m)}(n)=\int_{g(n)Z}F^{(m)}\,d\mu_{g(n)Z}$. Write the closure of $g(\Z)D$ as $P=Q/(Q\cap\Gamma^{\ell+1})$ and $Z=K\cdot e_W$, where $K$ is the normal closure of $\Delta(G^0)$ in $Q$, so that $\mu_{g(n)Z}=g(n)_*\mu_Z$ with $\mu_Z$ the $K$-invariant probability measure on $Z$. For $0\leq i\leq\ell$ let $\mathrm{p}_i:W\rightarrow X$ be the $i$-th coordinate projection. Then $(\mathrm{p}_i)_*\mu_Z$ is invariant under the $i$-th coordinate projection of $K$, which contains $G^0$. Since $X$ is connected, $G^0$ acts transitively on $X$, so $(\mathrm{p}_i)_*\mu_Z=\mu_{G/\Gamma}$, and hence $(\mathrm{p}_i)_*\mu_{g(n)Z}=(\alpha^{in})_*\mu_{G/\Gamma}=\mu_{G/\Gamma}$ for every $n$. With $C=\max_i\|f_i\|_\infty$, telescoping gives $|F^{(m)}-F^{(m')}|\leq C^{\ell}\sum_{i=0}^{\ell}|f_i^{(m)}-f_i^{(m')}|\circ\mathrm{p}_i$, and therefore
$$\sup_{n\in\Z}\bigl|s^{(m)}(n)-s^{(m')}(n)\bigr|\leq C^{\ell}\sum_{i=0}^{\ell}\bigl\|f_i^{(m)}-f_i^{(m')}\bigr\|_{L^1(\mu_{G/\Gamma})}\xrightarrow[m,m'\rightarrow\infty]{}0.$$
The same estimate for $\mu_{g(n)D}$, whose marginals are also $\mu_{G/\Gamma}$, shows that $a^{(m)}\rightarrow a$ uniformly. Hence $s^{(m)}\rightarrow s$ uniformly for some $s$, and $e^{(m)}=a^{(m)}-s^{(m)}\rightarrow e\coloneqq a-s$ uniformly. Each $s^{(m)}$ is a basic $r$-step nilsequence, so $s\in\NN_r\subseteq\NN_\ell$. Each $e^{(m)}$ is controlled by the system $(\Omega,\omega_0)$ that Lemma~\ref{auxiliary:lemma} gives for $P$ and $g$, with a null set independent of $m$ (Theorem~\ref{control:thm}). Hence $e$ is controlled by $(\Omega,\omega_0)$, and so $e\in\GenNz_\ell$ by Lemma~\ref{control:lem}. If $\XX$ is not connected, the same conclusion follows by passing to the connected components as in Section~\ref{sec:connectedness} below.
\subsection{Resolving Theorem~\ref{FK:thm} under two assumptions}
Let $\ell\geq 1$. In this section, we prove Theorem~\ref{FK:thm} for this $\ell$, under the following additional assumptions: we assume throughout that $\XX=(G/\Gamma,\mathcal{B},\mu_{G/\Gamma},R_\alpha)$ is connected and of step $\ell+1$. Then we fix any bounded $f_0,\dots,f_{\ell}\in L^\infty(\mu_{G/\Gamma})$.
\begin{definition}[Nilcharacter]
    A continuous function $f\in C(X)$ is called a \emph{nilcharacter} with frequency $\chi$ if $f(ux)=\chi(u)f(x)$ for all $u\in G_{\ell+1}$. Here $\chi:G_{\ell+1}\rightarrow S^1$ is a character which vanishes on $\Gamma\cap G_{\ell+1}.$
\end{definition}
It is well known that the algebra generated by the nilcharacters is dense in $C(X)$, and also in $L^2(X)$. As a result we may approximate our functions with linear combinations of nilcharacters and so it suffices to prove the theorem in the case where all the functions $f_0,\dots,f_{\ell}$ are nilcharacters. Also, if $f_1,\dots,f_\ell$ all have trivial frequency, then $a$ is a multiple correlation sequence on the connected $\ell$-step nilsystem $Z_\ell(X)$, treated above. We may therefore assume that some $f_i$ with $1\leq i\leq\ell$ has non-trivial frequency.

Assume throughout that $G=\langle G^0,\alpha\rangle$ (without loss of generality). Consider the space $\widehat{X} = (G\times G)/D(\Gamma\times \Gamma)$, where $D = \{(g,g):g\in G_{\ell+1}\}$ is the diagonal subgroup and equip $\widehat{X}$ with the Haar measure $\widehat{\mu}$ and the action $\widehat{T}$ induced by $R_\alpha\times R_\alpha$. Since each $f_i$ is a nilcharacter, the map $\widehat{f}_i(x,y) =f_i(x)\overline{f_i(y)}$ is well defined on $\widehat{X}$. Moreover, for $a(n)$ as above, we then have
$$|a(n)|^2 = \int_{\widehat{X}} \prod_{i=0}^{\ell} \widehat{f_i}(\widehat{T}^{in}x)\,d\widehat{\mu}(x).$$

Next, we embed $\imath:Z_{\ell}(X)\rightarrow \widehat{X}$ by $\imath(gG_{\ell+1}\Gamma) = (g\Gamma,g\Gamma)\mod D.$ Under this identification we get:
\begin{equation}\label{identification}
    \widehat{T}^n(e_{\widehat{X}}) = \imath(T_{\ell}^ne_{Z_{\ell}(X)}),
\end{equation}
for all $n\in \Z$, where $T_{\ell}$ is the induced action of $T$ on $Z_{\ell}(X)$. Now set
$$p_n=(T_\ell^n e_{Z_{\ell}(X)},T_\ell^{2n}e_{Z_{\ell}(X)},\dots, T_\ell^{\ell n}e_{Z_{\ell}(X)}),\quad \Omega=\overline{\{p_n:n\in\Z\}}\subseteq Z_{\ell}(X)^{\ell}.$$

Now, we apply the same argument as in the proof of Theorem~\ref{leibman:thm} with the nilmanifold $M=\widehat{X}^{\ell+1}$, the sub-nilmanifold $Y$ the diagonal in $M$ (i.e., $Y=\{(x,x,\dots,x)\in M : x\in \widehat{X}\}$) and the linear map $g(n)=(e,\widehat{\alpha}^n,\dots,\widehat{\alpha}^{\ell n})$, where $\widehat{\alpha}=(\alpha,\alpha)\in G\times G$ is the element inducing $\widehat{T}$. For $F(x_0,\dots,x_{\ell})=\prod_{i=0}^\ell \widehat{f_i}(x_i)$ we get:
$$|a(n)|^2 = \int_{g(n)Y} F\, d\mu_{g(n)Y}.$$
Here $\Omega\subseteq Z_{\ell}(X)^{\ell}$ is an ergodic $\ell$-step nilsystem (Lemma~\ref{lem:leibman-orbit}), and Theorem~\ref{control:thm} applies with $\eta=(\imath(e),\imath,\dots,\imath)$ by \eqref{identification} and with $F$ itself, which is an error term for $(Y,g)$ by \cite{BHK} whenever some $f_i$ with $1\leq i\leq\ell$ has a non-trivial frequency. Combining this with Lemma~\ref{control:lem} as in the proof of Theorem~\ref{leibman:thm}, the error term lies in $\GenNz_{\ell}$, as desired. 
\subsection{Removing the connectedness assumption}\label{sec:connectedness}
Suppose the step of $\XX$ is still $\ell+1$, and recall we have already established Theorem~\ref{FK:thm} for all connected nilsystems of this step. Let $\XX^0$ denote the connected component of the identity in $\XX$, by ergodicity $T$ acts cyclically on the connected components we may write 
$$\XX = \bigsqcup_{j=0}^{q-1} T^j \XX^0$$ where $q$ is the number of connected components in $\XX$. Let $\mu^0$ denote the normalized Haar measure on $\XX^0$ and let $S:\XX^0\rightarrow \XX^0$ denote the action $T^q$ induces on $\XX^0$. Then $(\XX^0,\mu^0,S)$ is an ergodic nilsystem, and we can also write $\XX^0 = G^0\Gamma/\Gamma$, where $G^0$ is the connected component of the identity in $G$. Let $a(n)$ be as above, and observe that $a(qn+r)$ for $0\leq r<q$ is a multiple correlation sequence on a connected component $T^r\XX^0$, pulling everything to $\XX^0$ gives
$$a(qn+r) = \frac{1}{q}\sum_{j=0}^{q-1} a_{j,r}(n),$$ where 
$$a_{j,r}(n)  = \int_{X^0} \prod_{i=0}^\ell S^{in} f_i^{j,r}\,d\mu^0,$$ where $f_i^{j,r}$ is the restriction of $T^{j+ir}f_i$ to $\XX^0$.

Now, all the multiple correlation sequences lives on $\XX^0$. We also need compatibility of the factors. By a classical result of Leibman \cite[Appendix A]{leibman-hkz} the Host--Kra factor of $T$ and $T^q$ coincide and $Z_{\ell}(X^0,S) = Z_\ell (X,T)^0.$ Therefore, if $h_i=\mathbb{E}(f_i\mid Z_{\ell}(X,T))$ then 
$$\mathbb{E}_{\mu^0}(f_i^{j,r}\mid Z_{\ell}(X^0,S))=(T^{j+ir}h_i)|_{X^0}.$$ Hence we get that $b(qn+r) = \frac{1}{q}\sum_{j=0}^{q-1}b_{j,r}(n)$ where $b=a-a_{\ell}$, $a_{\ell}(n)\coloneqq\int_X\prod_{i=0}^{\ell}h_i(T^{in}x)\,d\mu(x)$, and $b_{j,r}$ defined similarly to $a_{j,r}$.

Now we slightly modify the system from Lemma~\ref{auxiliary:lemma}. Let $M=Z_{\ell}(\XX)$, let $R$ denote the action induced by $T$, let $x_0:=\pi_\ell(e_X)$ where $\pi_\ell:\XX\rightarrow Z_\ell(\XX)$ is the factor map and consider the orbit $\omega_n:=(R^n x_0,R^{2n}x_0,\dots,R^{\ell n}x_0,n\mod q,n\theta)\in M^\ell\times \Z/q\Z\times \T.$ Let $\Omega=\overline{\{\omega_n:n\in \Z\}}$ be the orbit closure under the action 
$$T_\Omega = R\times R^2\times\dots\times R^{\ell}\times (+1)\times \theta.$$
By the same argument in the proof of Lemma~\ref{auxiliary:lemma}, $\Omega$ is an $\ell$-step nilsystem. Let $\Omega_0 =\overline{\{\omega_{qn}:n\in \Z\}}$, and $\Omega_r = T_\Omega^r \Omega_0$ for all $0\leq r <q$. Each $\Omega_r$ is clopen and of measure $\frac{1}{q}$. The maps $(z_1,\dots,z_\ell,0,t)\mapsto(z_1,\dots,z_\ell)$ and $(z_1,\dots,z_\ell,0,t)\mapsto((e,z_1,\dots,z_\ell),t)$ are factor maps from $(\Omega_0,T_\Omega^q)$ onto the systems used in the previous subsection for $(\XX^0,S)$ (taking the rotation $q\theta$ in Lemma~\ref{auxiliary:lemma}), so the null sets obtained there pull back to null subsets of $\Omega_0$. Hence, from the same arguments as before, for each tuple $(f_i^{j,r})_{i=0}^\ell$, we can find a $\mu_{\Omega_0}$-null set $E^{j,r}$. Taking $E_0\coloneqq\bigcup_{j,r} E^{j,r}\subseteq \Omega_0$, we see that whenever $|m_k|\rightarrow\infty$ and $\omega_{qm_k}\rightarrow\omega\in \Omega_0\backslash E_0$ we have $b_{j,r}(m_k)\rightarrow 0$ as $k\rightarrow\infty$ for all $j,r$ simultaneously.  Taking $E\coloneqq \bigcup_{r=0}^{q-1} T_{\Omega}^r E_0$ and let $|n_k|\rightarrow \infty$ be so that $\omega_{n_k}\rightarrow \omega\in \Omega\backslash E$ (this does not have to be the same $\omega$ from before). Then, write $n_k = qm_k + r_k$ for some $0\leq r_k<q$ and observe that we can have convergence only if $r_k=r$ eventually for some fixed $r$. Hence, $\omega_{qm_k}\rightarrow T_\Omega^{-r}\omega\not\in E_0$ and the previous argument gives that $b(n_k)\rightarrow 0$, as required. Applying the same arguments as before (Lemma~\ref{control:lem} and Theorem~\ref{control:thm}) gives the desired result. It remains to show that $a\in\NN_\ell+\GenNz_\ell$. We argue for $a$ directly, so that now the argument applies to any ergodic nilsystem $\XX$ of step at most $\ell+1$, and in particular to disconnected nilsystems of step at most $\ell$. By the connected case (the beginning of this section, and the previous subsection), each $a_{j,r}$ decomposes as $a_{j,r}=s_{j,r}+e_{j,r}$, where $s_{j,r}\in\NN_\ell$ and $e_{j,r}\in\GenNz_\ell$ is controlled by a null subset of $\Omega_0$ as above. Put $s(qn+r)=\frac{1}{q}\sum_{j=0}^{q-1}s_{j,r}(n)$ and $e(qn+r)=\frac{1}{q}\sum_{j=0}^{q-1}e_{j,r}(n)$, so that $a=s+e$. The error part $e$ is reassembled through $\Omega$ exactly as $b$ above, so $e\in\GenNz_\ell$. For the nilsequence part we use that if $c\in\ell^\infty(\Z)$ and $n\mapsto c(qn+r)$ belongs to $\NN_\ell$ for every $0\leq r<q$, then $c\in\NN_\ell$. For basic nilsequences this is \cite[Lemma~3.4]{Leibman1}: since all the sequences here are linear, the resulting polynomial nilsequence has degree $1$, so it is a basic $\ell$-step nilsequence by \cite[Proposition~3.3]{Leibman1}. For the general case, consider the map $I:(\ell^\infty(\Z))^q\rightarrow\ell^\infty(\Z)$, $I(c_0,\dots,c_{q-1})(qn+r)=c_r(n)$, which is linear and satisfies $\|I(c_0,\dots,c_{q-1})\|_\infty=\max_r\|c_r\|_\infty$. If every $c_r$ is a finite linear combination $c_r=\sum_{k=1}^K\lambda_{r,k}\psi_{r,k}$ of basic $\ell$-step nilsequences (with a common $K$, padding with $\psi=0$), then
$$I(c_0,\dots,c_{q-1})=\sum_{k=1}^K I(\lambda_{0,k}\psi_{0,k},\dots,\lambda_{q-1,k}\psi_{q-1,k}),$$
and each summand involves basic $\ell$-step nilsequences, so it lies in $\NN_\ell$ by the basic case. For arbitrary $c_r\in\NN_\ell$, choose such finite combinations $c_r^{(m)}\rightarrow c_r$ uniformly. Then $I(c^{(m)})\rightarrow I(c)$ uniformly, and $I(c)\in\NN_\ell$ since $\NN_\ell$ is closed. Hence $s\in\NN_\ell$, and so $a\in\NN_\ell+\GenNz_\ell$.
\subsection{Removing the step assumption via degree lowering}
In the previous section we have established Theorem~\ref{FK:thm} in the case where $\XX$ is an $(\ell+1)$-step nilsystem. Every step $r\leq \ell$ was treated at the beginning of this section. It is left to establish Theorem~\ref{FK:thm} for all steps. For this purpose, we adapt the \emph{degree lowering} mechanism from \cite{FranKuca}. We follow the terminology of Frantzikinakis and Kuca: let $\YY=(Y,\mathcal{C},\nu,T)$ be any system. We denote by $\|\cdot\|_{U^s(Y)}$ the Host--Kra seminorms on this system. Let $s\geq 1$, we say that sequences $q_1,\dots,q_\ell:\N\rightarrow \Z$ are \emph{good for degree $s$ seminorm control} for $\YY$ if 
$$\lim_{N\rightarrow\infty} \frac{1}{N} \sum_{n=1}^N T^{q_1(n)}f_1\cdots T^{q_{\ell}(n)}f_\ell =0$$ whenever there exist some $1\leq i \leq \ell$ with $\|f_i\|_{U^s(\YY)}=0.$ In other words, the limit behavior of the multiple ergodic average is determined by the projection of the function onto $Z_{s-1}(\YY)$. We say that $q_1,\dots,q_{\ell}:\N\rightarrow \Z$ are \emph{good for seminorm control} if they are good for degree $s$ seminorm control for some $s\in \N$. Let $d\geq 1$, we say that $q_1,\dots,q_{\ell}$ are \emph{good for $d$-step reduction} for $\YY$ if for every $Z_d(\YY)$-measurable bounded functions $f_1,\dots f_{\ell}$, we have
$$\lim_{N\rightarrow\infty} \frac{1}{N}\sum_{n=1}^N T^{q_1(n)}f_1\dots T^{q_{\ell}(n)}f_{\ell}=0$$ whenever $\|f_i\|_{U^d(\YY)}=0$ for some $1\leq i\leq \ell$ (recall that $\|f_i\|_{U^d(\YY)}=0\iff f_i\perp Z_{d-1}(\YY)$). The following was established in \cite[Theorem~3.11 and Proposition~6.1]{FranKuca}.
\begin{proposition}[Degree lowering]\label{degreelowering:thm}
Let $\YY=(Y,\mathcal{C},\nu,T)$ be a measure-preserving system. Let $\ell,d\geq 1$ and suppose that the sequences $q_1,\dots,q_{\ell}:\N\rightarrow \Z$ are good for seminorm control and $d$-step reduction. Then $q_1,\dots,q_{\ell}$ are good for degree $d$ seminorm control.
\end{proposition}

Now let $\XX$ be an ergodic $k$-step nilsystem, as in Theorem~\ref{FK:thm}, here with $k>\ell+1$. Set $\XX'=Z_{\ell+1}(\XX)$ and $W=Z_{\ell}(\XX)$. Let $\omega_n=(T_W^n e_W,\dots ,T_W^{\ell n}e_W,n\mod q, n\theta)$ and $\Omega = \overline{\{\omega_n:n\in \Z\}}\subseteq W^{\ell} \times \Z/q\mathbb{Z}\times \T$ as in the previous section. Applying the argument from the previous section to the $(\ell+1)$-step system $\XX'$, we can find a $\mu_{\Omega}$-null set $E\subseteq \Omega$ such that whenever $h_0,\dots, h_{\ell}$ are bounded and $\XX'$-measurable and one of which is orthogonal to $W$ we have:
\begin{equation}\label{hypothesis}|n_j|\rightarrow\infty\, \land\, \omega_{n_j}\rightarrow\omega\in \Omega\backslash E\Longrightarrow \lim_{j\rightarrow\infty}\int_{X'} h_0(x)\prod_{i=1}^{\ell} T^{in_j}h_i\,d\mu_{X'} =0.
\end{equation}
Crucially, the choice of $E$ is independent of $h_0,\dots,h_{\ell}$. To accomplish Theorem~\ref{FK:thm}, using the same arguments as before (Lemma~\ref{control:lem} and Theorem~\ref{control:thm}), we need to show that the same holds for arbitrary bounded $\XX$-measurable functions.

Fix $f_0,\dots f_{\ell}\in L^\infty(\mu)$, and suppose that for some $1\leq i_0\leq \ell$ we have $f_{i_0}\perp W$. Write $$a(n) = \int_X f_0(x)\cdot\prod_{i=1}^\ell f_i(T^{in}x)\,d\mu(x).$$ Let $\Omega,E$ exactly as constructed above and suppose by contradiction that there exists a sequence $|n_j|\rightarrow\infty$ and $\omega_{n_j}\rightarrow \omega\in\Omega\backslash E$ such that 
$a(n_j)$ fails to converge to zero. Passing to a subsequence, we may assume that \begin{equation}\label{contradiction2}
a(n_j)\rightarrow L\not=0.
\end{equation}

We wish to apply Theorem~\ref{degreelowering:thm}. However, unfortunately, our assumption in \eqref{hypothesis} only gives weak convergence of $\prod_{i=1}^{\ell} T^{in_j}h_i$, whereas we need norm convergence. The following lemma is therefore useful.
\begin{lemma}[From weak convergence to norm convergence of averages]
\label{lem:weak-to-norm}
Let $\XX=(X,\mathcal{B},\mu,T)$ be an invertible measure-preserving system, let $\ell\geq 1$, and let $|n_j|\rightarrow\infty$. Suppose that $h_1,\dots h_{\ell}$ are bounded functions satisfying 
$$\prod_{i=1}^\ell T^{in_j}h_i\rightarrow 0$$ weakly (i.e., the assumption in \eqref{hypothesis} holds). Then, there exists a subsequence $j_k\rightarrow\infty$ such that
$$\lim_{N\rightarrow\infty} \frac{1}{N}\sum_{k=1}^N \prod_{i=1}^\ell T^{in_{j_k}}h_i = 0$$ in norm.

Moreover, if the weak convergence hypothesis holds for each
tuple in a countable family
$$\mathcal{D} = \{(h_1^{(r)},\dots,h_\ell^{(r)}): r\geq 1\}\subseteq\bigr(L^\infty(\mu)\bigl)^\ell,$$
then the same subsequence $j_k$ can be chosen for the norm convergence. 
\end{lemma}

\begin{proof}
Since the claim for the countable-family assertion is stronger than the one for the single tuple, we shall only prove the former. For the sake of notation simplicity write $P_j^{(r)}\coloneqq \prod_{i=1}^\ell T^{in_j}h_i^{(r)}.$ Furthermore, by normalizing we shall assume that all the functions satisfy $\|h_i^{(r)}\|_{\infty}\leq 1.$ In particular, $$\|P_j^{(r)}\|_{2}\leq \|P_j^{(r)}\|_\infty\leq 1.$$ By assumption, $P_j^{(r)}\rightarrow 0$ weakly as $j\rightarrow\infty$ for all $r$.

From the uniform boundedness and the compactness of the unit ball in the weak$^*$-topology, we can construct inductively a sequence $j_k\rightarrow\infty$ such that
$$\left|\langle P^{(r)}_{j_k},P^{(r)}_{j_m}\rangle\right| \leq 2^{-k}$$ whenever $1\leq r\leq k$ and $1\leq m<k$. Indeed, suppose we have fixed $j_1,...,j_{k-1}$. Since $\langle P^{(r)}_{j},v\rangle\rightarrow 0$ by assumption for every $v=P^{(r)}_{j_m}$ that was already constructed, the existence of $j_k$ with the desired properties follow.

For a fixed $r$ we use the bound above whenever it is applicable and otherwise we use the $L^\infty$ bound. We get:
\begin{align*}\left\|\frac{1}{N} \sum_{k=1}^N P_{j_k}^{(r)}\right\|^2 &\leq \frac{1}{N}+\frac{2}{N^2}\sum_{1\leq a<b\leq N}\left|\langle P_{j_b}^{(r)},P_{j_a}^{(r)}\rangle\right| \\&\leq \frac{1}{N}+\frac{2}{N^2} \sum_{1\leq a<b<r} 1 + \frac{2}{N^2} \sum_{b=1}^\infty b2^{-b}\\&\longrightarrow0 \quad \text{as } N\rightarrow\infty. 
\end{align*}
\end{proof}
We return to the proof of Theorem~\ref{FK:thm}. Choose a countable set $\mathcal{D}_0\subseteq L^\infty(\XX')$ such that, for every $m\in\N$, $\mathcal{D}_0\cap\{\|f\|_\infty\leq m\}$ is $L^2$-dense in $\{f\in L^\infty(\XX'):\|f\|_\infty\leq m\}$ and write $\widetilde{\mathcal{D}_0}$ for all the elements of the form $f-\mathbb{E}(f\mid W)$ where $f\in \mathcal{D}_0$. Let $\mathcal{D}\subseteq (L^\infty(\XX'))^\ell$ denote the family of all $\ell$-tuples whose coordinates are in $\mathcal{D}_0\cup \widetilde{\mathcal{D}_0}$ and at least one of which is in $\widetilde{\mathcal{D}_0}$. Applying the previous lemma, we can find a subsequence $|n_{j_k}|\rightarrow\infty$ such that 
$$\lim_{N\rightarrow\infty} \frac{1}{N}\sum_{k=1}^N \prod_{i=1}^\ell T^{in_{j_k}}h_i = 0$$ in norm for all $(h_1,\dots,h_\ell)\in \mathcal{D}$. By density, the same is true for all $(h_1,\dots,h_\ell)\in \bigr(L^\infty(\XX')\bigr)^\ell$ with at least one being orthogonal to $W$. In other words, the sequences $(q_i(k) = i\cdot n_{j_k})_{i=1}^{\ell}$ are good for $(\ell+1)$-step reduction. Since $\XX$ itself is a nilsystem, they are also good for seminorm control and therefore good for degree $(\ell+1)$-seminorm control. Pairing with $f_0$ and using $f_{i_0}\perp W$, this implies that $\frac{1}{N}\sum_{k=1}^N a(n_{j_k})\rightarrow 0$. However, by \eqref{contradiction2}, $a(n_{j_k})\rightarrow L\neq 0$, so these averages converge to $L\neq0$, a contradiction.
We have shown that $a$ is controlled by $(\Omega,\omega_0)$, with the null set $E$, whenever $f_{i_0}\perp W$ for some $1\leq i_0\leq\ell$. To conclude the proof of Theorem~\ref{FK:thm}, let $f_0,\dots,f_\ell\in L^\infty(\mu)$ be arbitrary, and let $h_i=\mathbb E(f_i\mid W)$ and $a_\ell(n)=\int_X h_0\prod_{i=1}^{\ell}T^{in}h_i\,d\mu$. Since $\prod_{i\geq1}T^{in}h_i$ is $W$-measurable, telescoping gives
$$a-a_\ell=\sum_{i_0=1}^{\ell}a^{(i_0)},\qquad a^{(i_0)}(n)=\int_X f_0\prod_{i<i_0}T^{in}h_i\cdot T^{i_0n}(f_{i_0}-h_{i_0})\prod_{i>i_0}T^{in}f_i\,d\mu,$$
and in $a^{(i_0)}$ the $i_0$-th function is orthogonal to $W$. Hence, $a-a_\ell$ is controlled by $(\Omega,\omega_0)$. The sequence $a_\ell$ is a multiple correlation sequence on the ergodic $\ell$-step nilsystem $W$. Since $G_{\ell+1}\subseteq G^0$, the nilsystems $\XX$, $\XX'$ and $W$ have the same number $q$ of connected components, so the system constructed in Section~\ref{sec:connectedness} for $W$ is the system $\Omega$ above, and by that section $a_\ell=s+e$ with $s\in\NN_\ell$ and $e$ controlled by $(\Omega,\omega_0)$. Therefore $a=s+\bigl(e+(a-a_\ell)\bigr)$, where $s\in\NN_\ell$ and $e+(a-a_\ell)$ is controlled by $(\Omega,\omega_0)$, hence lies in $\GenNz_\ell$ by Lemma~\ref{control:lem}. Finally, $\NN_\ell+\GenNz_\ell\subseteq\GenI_\ell$ by Lemma~\ref{inclusions}(ii), since $\GenNz_\ell\subseteq\GenN_\ell$ and $\GenI_\ell$ is a linear space. This completes the proof of Theorem~\ref{FK:thm}.
\appendix
\section{The zeros of an analytic function and applications}\label{zero:section}
\begin{lemma}\label{zeros}
    Let $M$ be a compact nilmanifold with Haar measure $\mu_M$, let $(U,\phi)$ be a coordinate chart of $M$ with $U$ connected, and let $f:U\rightarrow \mathbb{C}$ be a real analytic non-zero function. Then its zero set is $\mu_M$-null. That is,
    $$\mu_M (\{u\in U : f(u)=0\})=0.$$
\end{lemma}
\begin{proof}
The function $f\circ\phi^{-1}$ is real analytic on the connected open set $\phi(U)\subseteq\R^{\dim M}$ and is not identically zero, so the same holds for its real part or its imaginary part. By \cite{Mityagin}, the zero set of $f\circ\phi^{-1}$ is therefore Lebesgue null. It remains to show that the Haar measure is absolutely continuous with respect to Lebesgue measure in local coordinates. This follows from \cite[Section~8]{Knapp}: by \cite[Theorems~8.21 and~8.23]{Knapp}, the Haar measure on $M$ is given by a nowhere-vanishing smooth volume form, and the coordinate integration formula (see \cite[Equation~(8.5a)]{Knapp}) gives
$$\mu_M(A) = \int_{\phi(A)}\rho(x)\, dx$$
for all Borel $A\subseteq U$, where $dx$ is the Lebesgue measure and $\rho$ is smooth and strictly positive (here $\rho$ is $F_\alpha$ as defined in \cite[Equation~(8.4)]{Knapp}). Hence the zero set of $f$ is $\mu_M$-null.
\end{proof}

\bibliographystyle{abbrv}
\bibliography{bibliography}

@book{KrantzParks,
  author    = {Krantz, Steven G. and Parks, Harold R.},
  title     = {A primer of real analytic functions},
  edition   = {Second},
  series    = {Birkh\"{a}user Advanced Texts: Basler Lehrb\"{u}cher},
  publisher = {Birkh\"{a}user Boston, Inc., Boston, MA},
  year      = {2002},
  pages     = {xiv+205},
  isbn      = {0-8176-4264-1},
}

@book{Hatcher,
  author    = {Hatcher, Allen},
  title     = {Algebraic topology},
  publisher = {Cambridge University Press, Cambridge},
  year      = {2002},
  pages     = {xii+544},
  isbn      = {0-521-79160-X; 0-521-79540-0},
}

@article {leibmanpolyseq,
    AUTHOR = {Leibman, A.},
     TITLE = {Polynomial sequences in groups},
   JOURNAL = {J. Algebra},
  FJOURNAL = {Journal of Algebra},
    VOLUME = {201},
      YEAR = {1998},
    NUMBER = {1},
     PAGES = {189--206},
      ISSN = {0021-8693,1090-266X},
   MRCLASS = {20F14},
  MRNUMBER = {1608723},
MRREVIEWER = {Mohammad\ Reza R. Moghaddam},
       DOI = {10.1006/jabr.1997.7269},
       URL = {https://doi.org/10.1006/jabr.1997.7269},
}

@book {Knapp,
    AUTHOR = {Knapp, Anthony W.},
     TITLE = {Lie groups beyond an introduction},
    SERIES = {Progress in Mathematics},
    VOLUME = {140},
   EDITION = {Second},
 PUBLISHER = {Birkh\"auser Boston, Inc., Boston, MA},
      YEAR = {2002},
     PAGES = {xviii+812},
      ISBN = {0-8176-4259-5},
   MRCLASS = {22-01},
  MRNUMBER = {1920389},
}

@article {Mityagin,
    AUTHOR = {Mityagin, B. S.},
     TITLE = {The zero set of a real analytic function},
   JOURNAL = {Mat. Zametki},
  FJOURNAL = {Matematicheskie Zametki},
    VOLUME = {107},
      YEAR = {2020},
    NUMBER = {3},
     PAGES = {473--475},
      ISSN = {0025-567X,2305-2880},
   MRCLASS = {26E05 (26B10)},
  MRNUMBER = {4070868},
MRREVIEWER = {Anna\ Valette-Stasica},
       DOI = {10.4213/mzm12620},
       URL = {https://doi.org/10.4213/mzm12620},
}

@book{Hel78,
  author    = {Helgason, Sigurdur},
  title     = {Differential Geometry, {L}ie Groups, and Symmetric Spaces},
  series    = {Pure and Applied Mathematics},
  volume    = {80},
  publisher = {Academic Press},
  address   = {New York--London},
  year      = {1978}
}

@article {LeibmanOrbit,
    AUTHOR = {Leibman, A.},
     TITLE = {Orbit of the diagonal in the power of a nilmanifold},
   JOURNAL = {Trans. Amer. Math. Soc.},
  FJOURNAL = {Transactions of the American Mathematical Society},
    VOLUME = {362},
      YEAR = {2010},
    NUMBER = {3},
     PAGES = {1619--1658},
      ISSN = {0002-9947,1088-6850},
   MRCLASS = {37A17 (22F30 37A05 37A30)},
  MRNUMBER = {2563743},
MRREVIEWER = {Nikos\ Frantzikinakis},
       DOI = {10.1090/S0002-9947-09-04961-7},
       URL = {https://doi.org/10.1090/S0002-9947-09-04961-7},
}

@article {nilpotentIPszemeredi,
    AUTHOR = {Zorin-Kranich, Pavel},
     TITLE = {A nilpotent {IP} polynomial multiple recurrence theorem},
   JOURNAL = {J. Anal. Math.},
  FJOURNAL = {Journal d'Analyse Math\'ematique},
    VOLUME = {123},
      YEAR = {2014},
     PAGES = {183--225},
      ISSN = {0021-7670,1565-8538},
   MRCLASS = {37A45 (11B75 28Dxx)},
  MRNUMBER = {3233579},
       DOI = {10.1007/s11854-014-0018-5},
       URL = {https://doi.org/10.1007/s11854-014-0018-5},
}

@article{DHJ3,
    AUTHOR = {Furstenberg, H. and Katznelson, Y.},
     TITLE = {A density version of the {H}ales-{J}ewett theorem for {$k=3$}},
      NOTE = {Graph theory and combinatorics (Cambridge, 1988)},
   JOURNAL = {Discrete Math.},
  FJOURNAL = {Discrete Mathematics},
    VOLUME = {75},
      YEAR = {1989},
    NUMBER = {1-3},
     PAGES = {227--241},
      ISSN = {0012-365X,1872-681X},
   MRCLASS = {05A05 (11B83 11K99 28D05)},
  MRNUMBER = {1001397},
MRREVIEWER = {R.\ L.\ Graham},
       DOI = {10.1016/0012-365X(89)90089-7},
       URL = {https://doi.org/10.1016/0012-365X(89)90089-7},
}

@article{DHJ,
    AUTHOR = {Furstenberg, H. and Katznelson, Y.},
     TITLE = {A density version of the {H}ales-{J}ewett theorem},
   JOURNAL = {J. Anal. Math.},
  FJOURNAL = {Journal d'Analyse Math\'ematique},
    VOLUME = {57},
      YEAR = {1991},
     PAGES = {64--119},
      ISSN = {0021-7670,1565-8538},
   MRCLASS = {28D05 (11B25)},
  MRNUMBER = {1191743},
MRREVIEWER = {B.\ Volkmann},
       DOI = {10.1007/BF03041066},
       URL = {https://doi.org/10.1007/BF03041066},
}

@article {Franpoly,
    AUTHOR = {Frantzikinakis, Nikos},
     TITLE = {Multiple ergodic averages for three polynomials and
              applications},
   JOURNAL = {Trans. Amer. Math. Soc.},
  FJOURNAL = {Transactions of the American Mathematical Society},
    VOLUME = {360},
      YEAR = {2008},
    NUMBER = {10},
     PAGES = {5435--5475},
      ISSN = {0002-9947,1088-6850},
   MRCLASS = {37A05 (11B75 28D05 37A45)},
       DOI = {10.1090/S0002-9947-08-04591-1},
       URL = {https://doi.org/10.1090/S0002-9947-08-04591-1},
}

@article{zieglerformula,
AUTHOR={Ziegler, Tamar},
Title = {A non-conventional ergodic theorem for a nilsystem.},
Journal = {Ergodic Theory and Dynamical Systems},
Volume = {25 no. 4},
Year = {2005},
Pages = {1357-1370.},
}

@article{BergelsonIPSzemeredi,

AUTHOR = {Bergelson, Vitaly and McCutcheon, Randall},
TITLE = {An Ergodic {IP} Polynomial {S}zemerédi Theorem},
Journal = {Memoirs of the AMS},
Volume = {146},
YEAR = {2000},
}

@misc{ShalomGrothendieck,
AUTHOR = {Or Shalom},
TITLE = {An application of Grothendieck theorem to the theory of multicorrelation sequences, multiple recurrence and partition regularity of quadratic equations},
Journal = {Preprint, available at https://arxiv.org/abs/2302.12857},
Year={2023},

}

@article {Stepin,
    AUTHOR = {Stepin, A. M.},
     TITLE = {Flows on solvable manifolds},
   JOURNAL = {Uspehi Mat. Nauk},
  FJOURNAL = {Akademiya Nauk SSSR i Moskovskoe Matematicheskoe Obshchestvo.
              Uspekhi Matematicheskikh Nauk},
    VOLUME = {24},
      YEAR = {1969},
    NUMBER = {5 (149)},
     PAGES = {241--242},
      ISSN = {0042-1316},
   MRCLASS = {28.70 (22.00)},
  MRNUMBER = {0267075},
MRREVIEWER = {J. Merza},
}

@article {HKM,
    AUTHOR = {Host, Bernard and Kra, Bryna and Maass, Alejandro},
     TITLE = {Complexity of nilsystems and systems lacking nilfactors},
   JOURNAL = {J. Anal. Math.},
  FJOURNAL = {Journal d'Analyse Math\'{e}matique},
    VOLUME = {124},
      YEAR = {2014},
     PAGES = {261--295},
      ISSN = {0021-7670},
   MRCLASS = {37A30 (37B05)},
  MRNUMBER = {3286054},
MRREVIEWER = {Siming Tu},
       DOI = {10.1007/s11854-014-0032-7},
       URL = {https://doi.org/10.1007/s11854-014-0032-7},
}

@article{ARS,
  Author        = {Ackelsberg, Ethan and Richter, Florian K. and Shalom, Or},
  Title         = {On the maximal spectral type of nilsystems},
  Journal       = {Proc. Amer. Math. Soc. Ser. B},
  Fjournal      = {Proceedings of the American Mathematical Society, Series B},
  Volume        = {11},
  Year          = {2024},
  Pages         = {469--480},
}

@article {ABS,
    AUTHOR = {Ackelsberg, Ethan and Bergelson, Vitaly and Shalom, Or},
     TITLE = {Khintchine-type recurrence for 3-point configurations},
   JOURNAL = {Forum Math. Sigma},
  FJOURNAL = {Forum of Mathematics. Sigma},
    VOLUME = {10},
      YEAR = {2022},
     PAGES = {Paper No. e107, 57},
      ISSN = {2050-5094},
   MRCLASS = {37A15 (05D10 37A30)},
  MRNUMBER = {4519061},
       DOI = {10.1017/fms.2022.97},
       URL = {https://doi.org/10.1017/fms.2022.97},
}

@article {morejoint,
    AUTHOR = {Donoso, Sebasti\'an and Ferr\'e{} Moragues, Andreu and
              Koutsogiannis, Andreas and Sun, Wenbo},
     TITLE = {Decomposition of multicorrelation sequences and joint
              ergodicity},
   JOURNAL = {Ergodic Theory Dynam. Systems},
  FJOURNAL = {Ergodic Theory and Dynamical Systems},
    VOLUME = {44},
      YEAR = {2024},
    NUMBER = {2},
     PAGES = {432--480},
      ISSN = {0143-3857,1469-4417},
   MRCLASS = {37A05 (28A99 37A30 60F99)},
  MRNUMBER = {4686270},
MRREVIEWER = {Anh\ N.\ Le},
       DOI = {10.1017/etds.2023.30},
       URL = {https://doi.org/10.1017/etds.2023.30},
}

@article {integerpart,
    AUTHOR = {Koutsogiannis, Andreas},
     TITLE = {Integer part polynomial correlation sequences},
   JOURNAL = {Ergodic Theory Dynam. Systems},
  FJOURNAL = {Ergodic Theory and Dynamical Systems},
    VOLUME = {38},
      YEAR = {2018},
    NUMBER = {4},
     PAGES = {1525--1542},
      ISSN = {0143-3857,1469-4417},
   MRCLASS = {37A30 (11B30 37A05 37A45 40A05 81P40 94A55)},
  MRNUMBER = {3789175},
MRREVIEWER = {Bryna\ Kra},
       DOI = {10.1017/etds.2016.67},
       URL = {https://doi.org/10.1017/etds.2016.67},
}

@article {MCS,
    AUTHOR = {Frantzikinakis, Nikos},
     TITLE = {Multiple correlation sequences and nilsequences},
   JOURNAL = {Invent. Math.},
  FJOURNAL = {Inventiones Mathematicae},
    VOLUME = {202},
      YEAR = {2015},
    NUMBER = {2},
     PAGES = {875--892},
      ISSN = {0020-9910,1432-1297},
   MRCLASS = {37A30 (05D10 11B30 37A05)},
  MRNUMBER = {3418246},
MRREVIEWER = {Vladimir\ S.\ Anashin},
       DOI = {10.1007/s00222-015-0579-7},
       URL = {https://doi.org/10.1007/s00222-015-0579-7},
}

@article {subsequences,
    AUTHOR = {Le, Anh Ngoc},
     TITLE = {Nilsequences and multiple correlations along subsequences},
   JOURNAL = {Ergodic Theory Dynam. Systems},
  FJOURNAL = {Ergodic Theory and Dynamical Systems},
    VOLUME = {40},
      YEAR = {2020},
    NUMBER = {6},
     PAGES = {1634--1654},
      ISSN = {0143-3857,1469-4417},
   MRCLASS = {37A44 (11N37 22E25 37A25 37A30)},
  MRNUMBER = {4092860},
MRREVIEWER = {Ryo\ Moore},
       DOI = {10.1017/etds.2018.110},
       URL = {https://doi.org/10.1017/etds.2018.110},
}

@article {decomposition,
    AUTHOR = {Le, Anh N. and Moreira, Joel and Richter, Florian K.},
     TITLE = {A decomposition of multicorrelation sequences for commuting
              transformations along primes},
   JOURNAL = {Discrete Anal.},
  FJOURNAL = {Discrete Analysis},
      YEAR = {2021},
     PAGES = {Paper No. 4, 27},
      ISSN = {2397-3129},
   MRCLASS = {37A44 (11B30 37A25)},
  MRNUMBER = {4274457},
MRREVIEWER = {Bryna\ Kra},
       DOI = {10.19086/da},
       URL = {https://doi.org/10.19086/da},
}

@article {FranKucajoint,
    AUTHOR = {Frantzikinakis, Nikos and Kuca, Borys},
     TITLE = {Joint ergodicity for commuting transformations and
              applications to polynomial sequences},
   JOURNAL = {Invent. Math.},
  FJOURNAL = {Inventiones Mathematicae},
    VOLUME = {239},
      YEAR = {2025},
    NUMBER = {2},
     PAGES = {621--706},
      ISSN = {0020-9910,1432-1297},
   MRCLASS = {37A44 (05D10 11B30 28D05)},
  MRNUMBER = {4850605},
MRREVIEWER = {Ryo\ Moore},
       DOI = {10.1007/s00222-024-01313-w},
       URL = {https://doi.org/10.1007/s00222-024-01313-w},
}

@article {FranKuca,
    AUTHOR = {Frantzikinakis, Nikos and Kuca, Borys},
     TITLE = {Degree lowering for ergodic averages along arithmetic
              progressions},
   JOURNAL = {J. Anal. Math.},
  FJOURNAL = {Journal d'Analyse Math\'ematique},
    VOLUME = {154},
      YEAR = {2024},
    NUMBER = {1},
     PAGES = {199--253},
      ISSN = {0021-7670,1565-8538},
   MRCLASS = {37A30},
  MRNUMBER = {4846312},
MRREVIEWER = {Song\ Shao},
       DOI = {10.1007/s11854-024-0347-y},
       URL = {https://doi.org/10.1007/s11854-024-0347-y},
}

@article {Fran,
    AUTHOR = {Frantzikinakis, Nikos},
     TITLE = {Some open problems on multiple ergodic averages},
   JOURNAL = {Bull. Hellenic Math. Soc.},
  FJOURNAL = {Bulletin of the Hellenic Mathematical Society},
    VOLUME = {60},
      YEAR = {2016},
     PAGES = {41--90},
   MRCLASS = {37A30 (05D10 11B30 37A45)},
  MRNUMBER = {3613710},
MRREVIEWER = {El Houcein El Abdalaoui},
}

@article {MC3,
    AUTHOR = {Ferr\'{e} Moragues, Andreu},
     TITLE = {Properties of multicorrelation sequences and large returns
              under some ergodicity assumptions},
   JOURNAL = {Discrete Contin. Dyn. Syst.},
  FJOURNAL = {Discrete and Continuous Dynamical Systems. Series A},
    VOLUME = {41},
      YEAR = {2021},
    NUMBER = {6},
     PAGES = {2809--2828},
      ISSN = {1078-0947},
   MRCLASS = {37A30 (37A15)},
  MRNUMBER = {4235630},
MRREVIEWER = {Donald Robertson},
       DOI = {10.3934/dcds.2020386},
       URL = {https://doi.org/10.3934/dcds.2020386},
}

@article {MC4,
    AUTHOR = {Koutsogiannis, Andreas and Le, Anh N. and Moreira, Joel and
              Richter, Florian K.},
     TITLE = {Structure of multicorrelation sequences with integer part
              polynomial iterates along primes},
   JOURNAL = {Proc. Amer. Math. Soc.},
  FJOURNAL = {Proceedings of the American Mathematical Society},
    VOLUME = {149},
      YEAR = {2021},
    NUMBER = {1},
     PAGES = {209--216},
      ISSN = {0002-9939},
   MRCLASS = {37A44 (11B30 37A15)},
  MRNUMBER = {4172598},
MRREVIEWER = {Joseph Andrew Vandehey},
       DOI = {10.1090/proc/15185},
       URL = {https://doi.org/10.1090/proc/15185},
}

@article {BrietGreen,
    AUTHOR = {Bri\"{e}t, Jop and Green, Ben},
     TITLE = {Multiple correlation sequences not approximable by
              nilsequences},
   JOURNAL = {Ergodic Theory Dynam. Systems},
  FJOURNAL = {Ergodic Theory and Dynamical Systems},
    VOLUME = {42},
      YEAR = {2022},
    NUMBER = {9},
     PAGES = {2711--2722},
      ISSN = {0143-3857},
   MRCLASS = {11B30 (11B25)},
  MRNUMBER = {4461688},
       DOI = {10.1017/etds.2021.66},
       URL = {https://doi.org/10.1017/etds.2021.66},
}

@article {BHK,
    AUTHOR = {Bergelson, Vitaly and Host, Bernard and Kra, Bryna},
     TITLE = {Multiple recurrence and nilsequences},
      NOTE = {With an appendix by Imre Ruzsa},
   JOURNAL = {Invent. Math.},
  FJOURNAL = {Inventiones Mathematicae},
    VOLUME = {160},
      YEAR = {2005},
    NUMBER = {2},
     PAGES = {261--303},
      ISSN = {0020-9910},
   MRCLASS = {37A30 (05D10 28D05 37A05)},
  MRNUMBER = {2138068},
MRREVIEWER = {Randall McCutcheon},
       DOI = {10.1007/s00222-004-0428-6},
       URL = {https://doi.org/10.1007/s00222-004-0428-6},
}

@article {Leibman2,
    AUTHOR = {Leibman, A.},
     TITLE = {Nilsequences, null-sequences, and multiple correlation
              sequences},
   JOURNAL = {Ergodic Theory Dynam. Systems},
  FJOURNAL = {Ergodic Theory and Dynamical Systems},
    VOLUME = {35},
      YEAR = {2015},
    NUMBER = {1},
     PAGES = {176--191},
      ISSN = {0143-3857},
   MRCLASS = {37A30 (28D05)},
  MRNUMBER = {3294297},
MRREVIEWER = {El Houcein El Abdalaoui},
       DOI = {10.1017/etds.2013.36},
       URL = {https://doi.org/10.1017/etds.2013.36},
}

@article {Leibman1,
    AUTHOR = {Leibman, A.},
     TITLE = {Multiple polynomial correlation sequences and nilsequences},
   JOURNAL = {Ergodic Theory Dynam. Systems},
  FJOURNAL = {Ergodic Theory and Dynamical Systems},
    VOLUME = {30},
      YEAR = {2010},
    NUMBER = {3},
     PAGES = {841--854},
      ISSN = {0143-3857},
   MRCLASS = {37A30 (28D05 37A05 37A45)},
  MRNUMBER = {2643713},
MRREVIEWER = {Bryna Kra},
       DOI = {10.1017/S0143385709000303},
       URL = {https://doi.org/10.1017/S0143385709000303},
}

@article{shalom2,
AUTHOR = {Shalom, Or}, 
TITLE= {Multiple ergodic averages in abelian groups and Khintchine type recurrence},
Journal = {Trans. Amer. Math. Soc. 375 (2022), 2729-2761}, 
Year = 2021
}

@book {hk-book,
    AUTHOR = {Host, Bernard and Kra, Bryna},
     TITLE = {Nilpotent structures in ergodic theory},
    SERIES = {Mathematical Surveys and Monographs},
    VOLUME = {236},
 PUBLISHER = {American Mathematical Society, Providence, RI},
      YEAR = {2018},
     PAGES = {X+427},
      ISBN = {978-1-4704-4780-9},
   MRCLASS = {37-02 (11B30 28D05 37Axx 37B05 47A35)},
  MRNUMBER = {3839640},
MRREVIEWER = {Yonatan Gutman},
       DOI = {10.1090/surv/236},
       URL = {https://doi.org/10.1090/surv/236},
}

@Article{austin2016nonconventional,
  Title                    = {{Non-conventional ergodic averages for several commuting actions of an amenable group}},
  Author                   = {T.~Austin},
  Journal                  = {J. Anal. Math.},
  Year                     = {2016},
  Pages                    = {243-274},
  Volume                   = {130}
}

@Article{bergelson1996polynomial,
  Title                    = {{Polynomial extensions of van der Waerden's and Szemer{\'e}di's theorems}},
  Author                   = {V.~Bergelson and A.~Leibman},
  Journal                  = {J.~Amer.~Math.~Soc.},
  Year                     = {1996},
  Number                   = {3},
  Pages                    = {725--753},
  Volume                   = {9}
}

@Article{furstenberg1977ergodic,
  Title                    = {{Ergodic behaviour of diagonal measures and a theorem of Szemer{\'e}di on arithmetic progressions}},
  Author                   = {H.~Furstenberg},
  Journal                  = {J. Anal. Math.},
  Year                     = {1977},
  Pages                    = {204-256},
  Volume                   = {31}
}

@Article{furstenberg1985ergodic,
  Title                    = {{An ergodic Szemer{\'e}di theorem for IP-systems and combinatorial theory}},
  Author                   = {H.~Furstenberg and Y.~Katznelson},
  Journal                  = {J. Anal. Math.},
  Year                     = {1985},
  Pages                    = {117--168}
}

@Article{furstenberg1978ergodic,
  Title                    = {{An ergodic Szemer{\'e}di theorem for commuting transformations}},
  Author                   = {H.~Furstenberg and Y.~Katznelson},
  Journal                  = {J. Anal. Math.},
  Year                     = {1978},
  Pages                    = {275--291},
  Volume                   = {34}
}

@Article{host2005nonconventional,
  Title                    = {{Nonconventional ergodic averages and nilmanifolds}},
  Author                   = {B.~Host and B.~Kra},
  Journal                  = {Ann.~Math.},
  Year                     = {2005},
  Number                   = {1},
  Pages                    = {397-488},
  Volume                   = {161}
}

@article {gt1,
    AUTHOR = {Green, Ben and Tao, Terence},
     TITLE = {The quantitative behaviour of polynomial orbits on
              nilmanifolds},
   JOURNAL = {Ann. of Math. (2)},
  FJOURNAL = {Annals of Mathematics. Second Series},
    VOLUME = {175},
      YEAR = {2012},
    NUMBER = {2},
     PAGES = {465--540},
      ISSN = {0003-486X},
   MRCLASS = {37A15},
  MRNUMBER = {2877065},
MRREVIEWER = {Tamar Ziegler},
       DOI = {10.4007/annals.2012.175.2.2},
       URL = {https://doi.org/10.4007/annals.2012.175.2.2},
}

@Article{robertson2016characteristic,
  Title                    = {{Characteristic factors for commuting actions of amenable groups}},
  Author                   = {D.~Robertson},
  Journal                  = {J. Anal. Math.},
  Year                     = {2016},
  Pages                    = {165-196},
  Volume                   = {129}
}

@Article{tao2008norm,
  Title                    = {{Norm convergence of multiple ergodic averages for commuting transformations}},
  Author                   = {Terence Tao},
  Journal                  = {Ergodic Theory Dynam. Systems},
  Year                     = {2008},
  Pages                    = {657-688},
  Volume                   = {28}
}

@Article{ziegler2007universal,
  Title                    = {{Universal characteristic factors and Furstenberg averages}},
  Author                   = {T.~Ziegler},
  Journal                  = {J.~Amer.~Math.~Soc.},
  Year                     = {2007},
  Pages                    = {53-97},
  Volume                   = {20}
}

@article {ABB,
    AUTHOR = {Ackelsberg, Ethan and Bergelson, Vitaly and Best, Andrew},
     TITLE = {Multiple recurrence and large intersections for abelian group
              actions},
   JOURNAL = {Discrete Anal.},
  FJOURNAL = {Discrete Analysis},
      YEAR = {2021},
     PAGES = {Paper No. 18, 91},
      ISSN = {2397-3129},
       DOI = {10.19086/da},
       URL = {https://doi.org/10.19086/da},
}

@incollection{leibman-hkz,
  author    = {Bergelson, V.},
  title     = {Combinatorial and {D}iophantine applications of ergodic theory},
  note      = {With appendices by A. Leibman and by A. Quas and M. Wierdl},
  booktitle = {Handbook of Dynamical Systems. {V}ol. 1{B}},
  editor    = {Hasselblatt, B. and Katok, A.},
  publisher = {Elsevier B. V., Amsterdam},
  year      = {2006},
  pages     = {745--869}
}

@article {Leibmanpointwise,
    AUTHOR = {Leibman, A.},
     TITLE = {Pointwise convergence of ergodic averages for polynomial
              sequences of translations on a nilmanifold},
   JOURNAL = {Ergodic Theory Dynam. Systems},
  FJOURNAL = {Ergodic Theory and Dynamical Systems},
    VOLUME = {25},
      YEAR = {2005},
    NUMBER = {1},
     PAGES = {201--213},
      ISSN = {0143-3857,1469-4417},
   MRCLASS = {37A17 (22F30 28D15)},
  MRNUMBER = {2122919},
MRREVIEWER = {Alexander\ Gorodnik},
       DOI = {10.1017/S0143385704000215},
       URL = {https://doi.org/10.1017/S0143385704000215},
}

@article {FrK052,
    AUTHOR = {Frantzikinakis, Nikos and Kra, Bryna},
     TITLE = {Ergodic averages for independent polynomials and applications},
   JOURNAL = {J. London Math. Soc. (2)},
  FJOURNAL = {Journal of the London Mathematical Society. Second Series},
    VOLUME = {74},
      YEAR = {2006},
    NUMBER = {1},
     PAGES = {131--142},
      ISSN = {0024-6107,1469-7750},
   MRCLASS = {37A45 (28D05 37A30)},
  MRNUMBER = {2254556},
MRREVIEWER = {Thomas\ Ward},
       DOI = {10.1112/S0024610706023374},
       URL = {https://doi.org/10.1112/S0024610706023374},
}

@misc{krashalom2025,
      title={Ergodic averages and the large intersection property along IP sets}, 
      author={Bryna Kra and Or Shalom},
      year={2025},
      eprint={2506.17771},
      archivePrefix={arXiv},
      primaryClass={math.DS},
      url={https://arxiv.org/abs/2506.17771}, 
}

@article {FrK05,
    AUTHOR = {Frantzikinakis, Nikos and Kra, Bryna},
     TITLE = {Polynomial averages converge to the product of integrals},
      NOTE = {Probability in mathematics},
   JOURNAL = {Israel J. Math.},
  FJOURNAL = {Israel Journal of Mathematics},
    VOLUME = {148},
      YEAR = {2005},
     PAGES = {267--276},
      ISSN = {0021-2172,1565-8511},
   MRCLASS = {37A25 (22E25 28D05 37A15 47A35)},
  MRNUMBER = {2191231},
MRREVIEWER = {Anish\ Ghosh},
       DOI = {10.1007/BF02775439},
       URL = {https://doi.org/10.1007/BF02775439},
}

@article {Fran-Hardy,
    AUTHOR = {Frantzikinakis, Nikos},
     TITLE = {Multiple recurrence and convergence for {H}ardy sequences of
              polynomial growth},
   JOURNAL = {J. Anal. Math.},
  FJOURNAL = {Journal d'Analyse Math\'ematique},
    VOLUME = {112},
      YEAR = {2010},
     PAGES = {79--135},
      ISSN = {0021-7670,1565-8538},
   MRCLASS = {47A35 (11C08 28D05 37A05)},
  MRNUMBER = {2762998},
MRREVIEWER = {Anthony\ Quas},
       DOI = {10.1007/s11854-010-0026-z},
       URL = {https://doi.org/10.1007/s11854-010-0026-z},
}

@misc{FKshalom,
  author = {Shalom, Or},
  title  = {Non-vanishing of multiple correlation sequences},
  year   = {2026},
  note   = {Preprint, \href{https://arxiv.org/abs/2607.13286}{arXiv:2607.13286}}
}

@article {Leibman3,
    AUTHOR = {Leibman, A.},
     TITLE = {Convergence of multiple ergodic averages along polynomials of
              several variables},
   JOURNAL = {Israel J. Math.},
  FJOURNAL = {Israel Journal of Mathematics},
    VOLUME = {146},
      YEAR = {2005},
     PAGES = {303--315},
      ISSN = {0021-2172,1565-8511},
   MRCLASS = {28D05 (47A35)},
  MRNUMBER = {2151605},
MRREVIEWER = {U.\ Krengel},
       DOI = {10.1007/BF02773538},
       URL = {https://doi.org/10.1007/BF02773538},
}
\end{document}